\documentclass{amsart}
\usepackage{mathrsfs}
\usepackage{color}
\usepackage{cleveref}
\usepackage{xypic}
\usepackage{verbatim}
\usepackage{MnSymbol}
\usepackage{tikz-cd}
\tikzset{>=latex}
\usepackage{bbold}
\xyoption{all}
\title{Topological Hecke eigenforms.}
\date{August 2020}
\author{L. Candelori}
\author{A. Salch}

\newtheorem{prop}{Proposition}[section]
\newtheorem{lemma}[prop]{Lemma}

\newtheorem{definition-proposition}[prop]{Definition-Proposition}
\newtheorem{theorem}[prop]{Theorem}
\newtheorem{corollary}[prop]{Corollary}

\theoremstyle{definition}
\newtheorem{remark}[prop]{Remark}
\newtheorem{definition}[prop]{Definition}
\newtheorem{example}[prop]{Example}
\newtheorem{observation}[prop]{Observation}

\renewcommand{\smash}{\wedge}

\DeclareMathOperator{\ev}{{\rm ev}}

\DeclareMathOperator{\pt}{{\rm pt.}}
\DeclareMathOperator{\op}{{\rm op}}

\DeclareMathOperator{\Spec}{{\rm Spec}}
\DeclareMathOperator{\tmf}{{\it tmf}}
\DeclareMathOperator{\Tmf}{{\it Tmf}}

\DeclareMathOperator{\im}{{\rm im}}
\DeclareMathOperator{\Ext}{{\rm Ext}}

\DeclareMathOperator{\Cotor}{{\rm Cotor}}
\DeclareMathOperator{\Ell}{{\it Ell}}
\DeclareMathOperator{\elll}{{\it ell}}

\DeclareMathOperator{\Mod}{{\rm Mod}}
\DeclareMathOperator{\cof}{{\rm cof}}
\DeclareMathOperator{\gr}{{\rm gr}}
\DeclareMathOperator{\et}{{\rm et}}
\DeclareMathOperator{\id}{{\rm id}}
\DeclareMathOperator{\Gal}{{\rm Gal}}
\DeclareMathOperator{\GL}{{\rm GL}}
\DeclareMathOperator{\Tr}{{\rm Tr}}
\DeclareMathOperator{\Frob}{{\rm Frob}}

\newcommand{\Qbar}{\overline{\mathbb{Q}}}

\newcommand{\tb}[1]{\textcolor{blue}{#1}}
\newcommand{\tr}[1]{\textcolor{red}{#1}}

\newcommand{\QQ}{\mathbb{Q}}

\newcommand{\ZZ}{\mathbb{Z}}

\begin{document}

\begin{abstract}
We study the eigenforms of the action of A. Baker's Hecke operators on the holomorphic elliptic homology of various topological spaces. 
We prove a multiplicity one theorem (i.e., one-dimensionality of the space of these ``topological Hecke eigenforms'' for any given eigencharacter) for some classes of topological spaces, and we give examples of finite CW-complexes for which multiplicity one fails. 
 We also develop some abstract ``derived eigentheory'' whose motivating examples arise from the failure of classical Hecke operators to commute with multiplication by various Eisenstein series. Part of this ``derived eigentheory'' is an identification of certain derived Hecke eigenforms as the obstructions to extending topological Hecke eigenforms from the top cell of a CW-complex to the rest of the CW-complex. Using these obstruction classes together with our multiplicity one theorem, we calculate the topological Hecke eigenforms explicitly, in terms of pairs of classical modular forms, on all $2$-cell CW complexes obtained by coning off an element in $\pi_n(S^m)$ which stably has Adams-Novikov filtration $1$.
\end{abstract}
\maketitle

\tableofcontents

\section{Introduction.}
\label{Introduction.}

In the study of classical modular forms, Hecke operators play a fundamental role: for example, if a normalized cusp form $f$ is an eigenform for the action of the Hecke operators, then the $n$th coefficient in the $q$-expansion of $f$ is equal to the eigenvalue of the Hecke operator $T_n$ on $f$. Therefore Hecke operators can be used to access the coefficients of interesting generating series (e.g. the number of points of a rational elliptic curve mod $p$) and can be used to construct Galois representations associated to modular forms. Much of modern algebraic number theory is concerned with the study of such Galois representations.

Meanwhile, elliptic homology $\Ell_*$ is a generalized homology theory with the property that, when evaluated on the $0$-sphere, $\Ell_*(S^0)$ is the classical graded ring of weakly holomorphic modular forms over $\mathbb{Z}[\frac{1}{6}]$ of level one. 
In \cite{MR1037690}, A. Baker constructed certain operations $\tilde{T}_2, \tilde{T}_3, \tilde{T}_4,\tilde{T}_5, \dots$ on elliptic homology, with the property that, when evaluated on $S^0$, Baker's operations agree with the classical Hecke operators $T_2, T_3, T_4, T_5, \dots$ on modular forms. A straightforward modification of Baker's argument yields ``topological Hecke operators'' on holomorphic elliptic homology $\elll_*$, which, when evaluated on the zero-sphere $S^0$, agree with the classical Hecke action on holomorphic modular forms over $\mathbb{Z}[\frac{1}{6}]$ of level $1$.
Naturally, one wants to know how these ``topological Hecke operators'' act on the holomorphic elliptic homology of topological spaces other than spheres\footnote{In \cite{MR1037690}, Baker remarks ``Because our operations are merely additive (and not multiplicative...) they appear to be hard to compute explicitly except in a few simple situations.'' In A. Ranicki's review of \cite{MR1037690} on the AMS's Mathematical Reviews service, Ranicki includes that quotation from Baker's paper, evidently to emphasize that it indeed seems difficult to calculate Baker's topological Hecke operators. Part of the purpose of this paper is to remedy this situation, by demonstrating how to calculate the action of Baker's topological Hecke operators on the elliptic homology of various CW-complexes---specifically, those whose attaching maps are identifiable as classes in the Adams-Novikov spectral sequence with known cocycle representatives in the cobar complex for Brown-Peterson homology---and giving explicit results in a certain reasonable class of examples.}, and in particular, one wants a calculation of the eigenforms of the topological Hecke action on the elliptic homology of various topological spaces $X$. Ideally, such a calculation would let us translate between topological properties of $X$ and number-theoretic properties of the topological Hecke eigenforms over $X$.

We carry out such a calculation in this paper. Here is a synopsis of the main ideas and results:
\begin{itemize}

\item \Cref{Derived eigentheory.} is purely algebraic, and offers a kind of ``derived eigentheory.'' The idea is that, if a ring $A$ acts on some module $V$, then for some ring homomorphism $\lambda: A \rightarrow \mathcal{O}_K$ into a ring of algebraic integers (i.e. an {\em eigencharacter of $A$}), the eigenvectors for the action of $A$ on $V$ with eigencharacter $\lambda$ are given by a certain Hochschild cohomology group $HH^0(A; V^{\lambda})$ defined in Definition \ref{bimodules definition 1}. We then define the ``derived eigenvectors'' with eigencharacter $\lambda$ as the elements of $HH^n(A; V^{\lambda})$ for $n>0$, although in this paper we are almost exclusively concerned with the case $n=1$. 
We present this theory of derived eigenvectors in a general algebraic setting, although its cogent case is when $A$ is an abstract Hecke algebra of level one Hecke operators and $V$ a graded ring of level one holomorphic modular forms. 

Definition-Proposition \ref{def of dot-cup product} produces certain operations on derived eigenvectors, given by taking the ``dot-cup'' product with certain elements of $HH^n(A; \hom_R(M_*,M_*))$. Natural elements of $HH^1(A; \hom_R(M_*,M_*))$ are produced in Definition-Propositions \ref{def of delta} and \ref{def of kappas}. Specializing to the case of modular forms, for each odd prime $p$ and each positive integer $n$, we get a cohomology class $\kappa^{E_{p-1}}_n\in HH^1(A; \hom_R(M_*,M_*))$ which measures the failure of each Hecke operator $\tilde{T}_{\ell}$ to commute with the $R$-module map given by multiplication by the $n$th power of the Eisenstein series $E_{p-1}$. In Theorem \ref{nontriviality of all kappas}, we show that $\kappa^{E_{p-1}}_n$ is in general nontrivial, and generates a subgroup of $HH^1(A; \hom_R(M_*,M_*))$ of order $p^{1+\nu_p(n)}$, where $\nu_p$ denotes $p$-adic valuation.

\item In \cref{Review of established...} we give a very brief review of holomorphic and weakly holomorphic elliptic homology, their relations to various versions of the spectrum of topological modular forms, and Baker's topological Hecke operations on them.

\item In \cref{Topological Hecke eigenforms and...}, we define topological Hecke eigenforms, and the multiplicity one condition for topological Hecke eigenforms. Proposition \ref{multiplicity one is convenient} shows that, if $X$ has multiplicity one for some set of primes $P$ and $f\in \elll_{2k}(X)\otimes_{\mathbb{Z}}\mathcal{O}_{\overline{\mathbb{Q}}}[P^{-1}]$ is an eigenform for the action of $\tilde{T}_{\ell}$ for all primes $\ell\in P$, then $f$ is an eigenform for the action of $\tilde{T}_n$ for all positive integers $n$ whose prime factors are in $P$. 

Multiplicity one is a well-known result for classical level $1$ modular forms, but its topological analogue holds over some topological spaces and does not hold over others. One reason why this matters is as follows. In general, in the absence of a multiplicity one result, an element of elliptic homology which is an eigenform for the action of $\tilde{T}_{\ell}$ for all primes $\ell\in P$ is not guaranteed to be an eigenform for the action of $\tilde{T}_{\ell^2}$, for example, since the topological Hecke operators satisfy $\tilde{T}_{\ell^{r+2}} = \tilde{T}_{\ell}\tilde{T}_{\ell^{r+2}} - \frac{1}{\ell}\Psi^{\ell}\tilde{T}_{\ell^r}$, and $f$ might not be an eigenform for the action of the Adams operation $\Psi^{\ell}$. This subtlety with Adams operations does not arise in the classical number-theoretic context of modular forms, where $\Psi^{\ell}$ always acts on weight $k$ modular forms simply as multiplication by $\ell^k$.
\item
Theorem \ref{multiplicity one thm} establishes that, if $X$ is a finite CW-complex with torsion-free homology and at most one cell in each dimension, and $P$ is a cofinite set of primes, then $X$ has the multiplicity one property for Hecke operators whose prime factors are in $P$. 
\item In Proposition \ref{examples of hecke eigenforms}, we calculate the topological Hecke eigenforms over any finite wedge product of spheres. Suppose that $X$ is a finite wedge product of spheres and $P$ is a cofinite set of primes (i.e, $P$ is the complement of some finite set of primes), and write $V_X(\lambda)_*$ for the graded $\mathbb{Z}[P^{-1}]$-module of topological Hecke eigenforms over $X$, for all primes in $P$, with eigencharacter $\lambda$.
Write $V(\lambda)_*$ for the graded $\mathbb{Z}[P^{-1}]$-submodule of $M_*\otimes_{\mathbb{Z}} \mathbb{Z}[P^{-1}]$ consisting of eigenforms for the action of $T_p$, for all $p\in P$, with eigencharacter $\lambda$. Write $D(V(\lambda)_*)$ for $V(\lambda)_*$ with all grading degrees doubled. Proposition \ref{examples of hecke eigenforms} shows that
we have an isomorphism of graded $\mathbb{Z}[P^{-1}]$-modules
$V_X(\lambda)_* \cong D(V(\lambda)_{*})\otimes_{\mathbb{Z}} \tilde{H}_{*}(X;\mathbb{Z})\otimes_{\mathbb{Z}} \mathbb{Z}[P^{-1}]$.
\item In \cref{Extending a Hecke eigenform...}, we study the problem of extending a topological Hecke eigenform from the top cell of a $2$-cell CW-complex to a topological Hecke eigenform on the whole complex. (The opposite problem, of extending a topological Hecke eigenform from the bottom cell to the whole complex, is trivial: every topological Hecke eigenform on the bottom cell extends uniquely to the whole complex.) In particular, in Corollary \ref{main thm cor 123409}, we show that, if $X$ is the homotopy cofiber of the map $p^jv_1^n\alpha_1 \in \pi_{(2p-2)(n+1) - 1}(S^0)$ for some prime $p$ and positive integer $n$ and some nonnegative integer $j$ (see the next paragraph for some explanation of what $v_1^n\alpha_1$ is), and if $f$ is a topological Hecke eigenform on the top cell of $X$, then $f$ extends to a topological Hecke eigenform on all of $X$ if and only if the derived Hecke eigenform $f\cupdot p^j \kappa^{E_{p-1}}_n\in HH^1\left(A; M_{k+(p-1)n}^{\lambda}\right)$ is trivial. Here $\cupdot$ is the ``dot-cup'' product on Hochschild cohomology, defined in Definition \ref{def of dot-cup product}, and the cohomology class $\kappa^{E_{p-1}}_n$ is the one described above, which measures the failure of Hecke operators to commute with multiplication by $E_{p-1}^n$. 
\item
Theorem \ref{main thm 304jf} then uses the obstruction theory and the multiplicity one theorems described above, together with a Hochschild cohomology calculation carried out in Proposition \ref{nonvanishing of dotcup 0}, to 
give a calculation of all topological Hecke eigenforms over $2$-cell CW-complexes of a certain natural type. This ``natural type'' deserves some explanation. Every connected CW-complex $X$ with two cells is built by attaching an $n$-cell to an $m$-cell with $n\geq m$; the homotopy type of that CW-complex is determined by the homotopy class of the map from the boundary $\delta e^n \cong S^{n-1}$ of the $n$-cell $e^n$ to the $m$-cell, whose boundary has been collapsed to the $m-1$-skeleton, which must be a single point\footnote{By convention, we do not count the basepoint as a $0$-cell when we say that a $2$-cell complex has ``two cells.''}. Consequently, the homotopy type of $X$ is determined by an element of the $(n-1)$st homotopy group $\pi_{n-1}(S^m)$ of the $m$-sphere. Since elliptic homology is a generalized homology theory and since Baker's topological Hecke operators are stable operators, the topological Hecke operators on a space depend only on its {\em stable} homotopy type. The stable homotopy type of $X$ depends only on the stable homotopy class of the attaching map $\delta e^n \rightarrow S^m$, so $X$ is determined by an element of $\pi_{n-1-m}^{st}(S^0)$, the stable homotopy groups\footnote{Since Baker's topological Hecke operators are stable operations, we are free to use methods from stable homotopy throughout this paper. Consequently, in the rest of this paper after the introduction, we drop the superscript ${}^{st}$ to indicate stability, and we simply write $\pi_*$ for {\em stable} homotopy groups.} of spheres. 

The most familiar of the stable homotopy groups $\pi_*^{st}(S^0)$ of spheres is the zeroth homotopy group $\pi_0^{st}(S^0)\cong \mathbb{Z}$: attaching a cell to $S^m$ by a map whose stable class is some nonzero integer $d\in\mathbb{Z}\cong \pi_0^{st}(S^0)$ simply cones off the degree $d$ map on $S^m$, yielding a Moore space. The holomorphic elliptic homology of the resulting space is the modulo $d$ reduction of the ring of holomorphic modular forms over $\mathbb{Z}[\frac{1}{6}]$, and the topological Hecke operators on it agree with the modulo $d$ reduction of the classical Hecke operators. Consequently, the case where our attaching map lives stably in $\pi_0^{st}(S^0)$ is relatively unmysterious.

Consequently we focus on the next nontrivial case, that is, the case where the attaching map is an integer multiple of $v_1^m\alpha_1$ for some\footnote{Both $v_1\in \pi_{2p-2}(BP)$ and $\alpha_1\in \pi_{2p-3}^{st}(S^0)$ depend on the choice of prime $p$, but the prime $p$ is suppressed from the notation for $v_1$ and $\alpha_1$.} prime $p$ and some nonnegative integer $m$. Here is some explanation: if $n>0$, then the $n$th stable homotopy group $\pi_n^{st}(S^0)$ of the zero-sphere is known to be a {\em finite} abelian group, hence decomposes as the direct sum, across all primes $p$, of its $p$-localizations. Consequently the stable homotopy groups of spheres are usually studied $p$-locally, at each prime $p$. A very common approach to the $p$-local stable homotopy groups of spheres is to filter them by the $BP$-Adams filtration\footnote{$BP$ denotes $p$-local Brown-Peterson homology, which depends on the prime $p$, but the prime $p$ is traditionally suppressed from the notation $BP$. A standard reference is the book \cite{MR860042}.}, which arises naturally from the $p$-local Adams-Novikov spectral sequence; one way to describe the $p$-local Adams-Novikov spectral sequence is that it is a spectral sequence whose $E_2$-term is the flat cohomology $H^*_{fl}(\mathcal{M}_{fg}; \omega^*)$ of the moduli stack of one-dimensional formal groups over $\Spec \mathbb{Z}_{(p)}$, and which converges to the $p$-localization $\left( \pi_*^{st}(S^0)\right)_{(p)}$ of the stable homotopy groups of spheres. An element of $\left(\pi_*^{st}(S^0)\right)_{(p)}$ then is said to have $BP$-Adams degree $n$ if, in the spectral sequence, it is detected by an element of $H^n_{fl}(\mathcal{M}_{fg}; \omega^*)$. For all primes $p$, the only elements of $BP$-Adams degree $0$ are the elements of $\left(\pi_0^{st}(S^0)\right)_{(p)}\cong \mathbb{Z}_{(p)}$, which we already described. At odd\footnote{Since $2$ and $3$ are already inverted in elliptic homology, we need only consider odd primes.} primes $p$, there are no elements of $\pi_n^{st}(S^0)$ of $BP$-Adams degree $1$ unless $n\equiv -1$ modulo $2p-2$, in which case we have a cyclic subgroup of $\pi_{(2p-2)(m+1) -1}^{st}(S^0)$ of order $p^{1+\nu_p(m)}$ which has $BP$-Adams degree $1$. The element named $v_1^{m}\alpha_1$ is a relatively standard choice of generator for this cyclic subgroup of $\pi_{(2p-2)(m+1) -1}^{st}(S^0)$. Readers interested in further information about the stable homotopy groups of spheres and the Adams-Novikov spectral sequence should consult the standard reference \cite{MR860042}.

With this in mind, the next natural class of CW-complexes whose topological Hecke eigenforms ought to be calculated are those $2$-cell complexes obtained by attaching a cell via an element of $BP$-Adams degree $1$ for some odd prime $p$, i.e., attaching a cell via an integer multiple of $v_1^m\alpha_1$ for some nonnegative integer $m$ and some odd prime $p$.
Theorem \ref{main thm 304jf} accomplishes that task:
the topological Hecke eigenforms over any two-cell CW-complex whose attaching map has stable homotopy class $p^jv_1^m\alpha_1\in \pi_{2(p-1)(m+1)-1}^{st}(S^0)$ are of exactly two types: 
\begin{itemize}
\item those supported on the bottom cell, of which there is exactly one, with eigencharacter $\lambda$, for each classical Hecke eigenform with eigencharacter $\lambda$; and
\item those nontrivial on the top cell, of which there is exactly one, with eigencharacter $\lambda$, for each classical Hecke eigenform $g$ {\em which is divisible by } $p^{1+\nu_p(m)-j}$ and which has eigencharacter $\lambda$. The component on the bottom cell is then the holomorphic modular form (not generally a Hecke eigenform!) $-p^{j-1-\nu_p(m)}gE_{p-1}^n$. 
\end{itemize}
The appearance of a divisibility condition here is not surprising: the coefficient ring of holomorphic elliptic homology is the ring of holomorphic modular forms {\em over $\mathbb{Z}[\frac{1}{6}]$}, so not every element is divisible by every integer; and if we eliminate questions of divisibility by inverting all primes, then all the nontrivial topological information is lost, since inverting all primes also kills off the stable homotopy classes of all attaching maps in $\pi_n(S^m)$ with $n>m$, and the resulting ``rational topological Hecke eigenforms'' would simply be the same as on a wedge of spheres (which are described completely in Proposition \ref{examples of hecke eigenforms}).
\end{itemize}

Throughout this paper, we confine our attention to elliptic homology and modular forms of level $1$. Generalizations to other levels are possible, in some cases trivially, and in some cases quite nontrivially. In particular, multiplicity one results like Theorem \ref{multiplicity one thm} for level $>1$ require some care, since in the classical case for level $>1$ one only has multiplicity one for cuspidal {\em newforms}. Formulating a suitable treatment of topological newforms and topological oldforms is beyond the scope of this paper, although it ought to be doable: for example, the orthogonality of newforms and oldforms under the Petersson product classically plays a fundamental role in the subject, and we make an effort to lay some kind of groundwork for a topological Petersson product in \cite{peterssonproductpreprint}.

In the present paper we have tried to strike a balance between being readable and interesting to number theorists, and being readable and interesting to topologists. We hope we have succeeded, and we apologize to readers who may be unhappy with any compromises between these two audiences which we have made in our exposition.

We are grateful to Jack Davies for timely and insightful comments which improved this paper. The first author acknowledges support by the U.S. Department of Energy, Office of Science, Basic Energy Sciences, under Award Number DE-SC-SC0022134.

\section{Derived eigentheory.}
\label{Derived eigentheory.}

\subsection{Abstract calculus of commutators of operations.}
\label{Abstract calculus....}
Let $R$ be a commutative ring, $A$ a commutative $R$-algebra, and let $M$ be a graded $A$-module equipped with the structure of a commutative graded $R$-algebra . Of course, the action of a given element $T\in A$ on $M$ may or may not commute with multiplication by a given homogeneous element $E$ of $M$; that is, we may or may not have $T(Em) = T(m)E$ for all homogeneous $m\in M$. Starting from this observation, we show in this section that it is possible to construct interesting non-trivial Hochschild 1-cocycles for the algebra $A$. Our motivating example is when $A$ is the ``abstract'' Hecke algebra, defined below:

\begin{definition}
\label{def: abstract hecke algebra}
Let  $\Pi \subseteq \ZZ$ be the set of all prime numbers, and let $\{T_p\}_{p\in \Pi} = \{T_2, T_3, T_5, T_7, T_{11}, \ldots\}$ be a set of indeterminate variables indexed by the primes. The {\em abstract Hecke algebra} $\mathbf{A}$ is the polynomial $\ZZ$-algebra 
$$
\mathbf{A} := \ZZ[\{T_p\}_{p\in \Pi}] = \ZZ[T_2, T_3, T_5, T_7, T_{11}, \ldots ].
$$
\end{definition}

Let $R$ be $\mathbb{Z}$ or a localization of $\mathbb{Z}$. For any $k \in \ZZ_{\geq 0}$, let $M_k$ be the $R$-module of weight $k$, level one, holomorphic modular forms with Fourier coefficients in $R$. Let 
$
M_* = \bigoplus_{k\in \ZZ_{\geq 0}} M_k
$
be the graded ring of holomorphic modular forms over $R$. Let $\mathbf{A}$ act on $M_*$ as follows: if $f_k = \sum_{n=0}^{\infty} a_n q^n \in M_k$, then the action of $T_p$ on $f_k$ is by the usual action of Hecke operators on weight $k$ modular forms, 
$$
T_p(\sum_{n=0}^{\infty} a_n q^n) = \sum_{n=0}^{\infty} a_{pn}q^n + p^{k-1}\sum_{n=0}^{\infty} a_n q^{np}.
$$
\begin{remark}
For each weight $k$, there are also operators $T_n$ for all $n\in \ZZ_{\geq 0}$ (not necessarily prime), defined recursively in terms of the $T_p$, $p\in \Pi$, by the Hecke relations
\begin{align}
\label{hecke rel 1} T_{mn} &= T_mT_n\quad (m,n) = 1 \\
\label{hecke rel 2} T_{p^{r+2}} &= T_pT_{p^{r+1}} - p^{k-1}T_{p^r} \quad p \text{ prime }, r\geq 0.
\end{align} 
\end{remark}

Fix a prime $p>3$, and let 
$$
E_{p-1} = 1 - \frac{2(p-1)}{B_{p-1}}\sum_{n=1}^{\infty} \sigma_{p-2}(n)q^n \in M_{p-1}
$$
be the Eisenstein series of weight $p-1$. It satisfies the following well-known integrality property: 

\begin{prop}
\label{prop: p-integrality}
$E_{p-1}$ is $p$-integral, that is, it belongs to $M_*\otimes_{\ZZ} \ZZ_{(p)}$ and its reduction mod $p$ satisfies $E_{p-1} \equiv 1 \mod p$. 
\end{prop}

\begin{proof}
By the von Staudt-Clausen Theorem, 
$$
B_{p-1} + \sum_{\substack {q \text{ prime } \\ (q-1) \mid (p-1)}} \frac{1}{q}  
\in \ZZ 
$$
which implies that $pB_{p-1} \in 1 + p\ZZ_{(p)}$. Therefore 
$$
0 = v_p( p B_{p-1}) = v_p(p) + v_p(B_{p-1}) = 1 + v_p(B_{p-1}),
$$
that is, $v_p(B_{p-1}) = -1$ and therefore the entire expression $\frac{2(p-1)}{B_{p-1}} \sum_{n=1}^{\infty} \sigma_{p-2}(n) q^n$ is divisible by $p$. 
\end{proof}

Going back to our original problem, with $A= \mathbf{A}, M = M_*$, and $E = E_{p-1}$, we do not generally have $T_{\ell}(E_{p-1}f) = T_{\ell}(f)E_{p-1}$ for all modular forms $f\in M_*$ and primes $\ell$. However, since $E_{p-1} \equiv 1 \mod p$ by Proposition \ref{prop: p-integrality}, we have $T_{\ell}(E_{p-1}f) \equiv T_{\ell}(f)E_{p-1}$ mod $p$. This suggests that perhaps there is more to be said about the relationship between $T_{\ell}(E_{p-1}f)$ and $T_{\ell}(f)E_{p-1}$.

The following lemma will be important for the proof of Proposition \ref{interpretation of Dn}, and later, in Theorem \ref{eigenform conditions thm}. The proof is extremely elementary, but we include it for the sake of completeness.
\begin{lemma}\label{binomial coeffs lemma}
Let $k,n$ be integers such that $0\leq k\leq n$. Then
\begin{align*}
 \sum_{i=0}^{n-k}(-1)^i\binom{n}{i+k}\binom{i+k}{i}
  &= \left\{\begin{array}{ll} 0 &\mbox{ if\ }k<n \\ 1 &\mbox{ if\ }k=n .\end{array}\right.
\end{align*}
\end{lemma}
\begin{proof}
Given a set $S$ with $n$ elements, the product $\binom{n}{i+k}\binom{i+k}{i}$ is the number of choices of an $(i+k)$-element subset $T$ of $S$ together with an $i$-element subset $U$ of $T$. Such a pair $(T,U)$ is equivalently specified, uniquely, by giving a $k$-element subset $V$ of $S$ and a $i$-element subset $U$ of $S$ which is disjoint from $V$. We then have $T = V\cup U$ and $V = T\backslash U$. The number of such pairs $(V,U)$ is $\binom{n}{k}\binom{n-k}{i}$. Consequently we have equalities
\begin{align*}
 \sum_{i=0}^{n-k}(-1)^i\binom{n}{i+k}\binom{i+k}{i} 
  &= \sum_{i=0}^{n-k}(-1)^i\binom{n}{k}\binom{n-k}{i}\\
  &= \binom{n}{k}\sum_{i=0}^{n-k}(-1)^i\binom{n-k}{i}\\
  &= \left\{\begin{array}{ll} \binom{n}{k}\cdot 1 &\mbox{ if\ }n-k=0 \\ \binom{n}{k}\cdot 0 &\mbox{ if\ }n-k>0.\end{array}\right.
\end{align*}
\end{proof}

For the moment, we return to the general situation of a commutative $R$-algebra $A$, a graded $A$-module $M$ with $R$-algebra structure, and a homogeneous element $E\in M$. The action of each element $T\in A$ defines an element of $\hom_{\Mod(R)}(M,M)$, which we will also call $T$; and multiplication by $E$ also defines an element of $\hom_{\Mod(R)}(M,M)$, which we will call $e$.
Under composition, $\hom_{\Mod(R)}(M,M)$ is an associative $R$-algebra. Suppose, as in the case of $T_{\ell}$ and $E_{p-1}$, that there exists some element $\mu\in R$ such that the commutator of $T$ and $e$ is congruent to zero modulo $\mu$. Suppose that $M$ is $\mu$-torsion-free, and let $T'\in \hom_{\Mod(R)}(M,M)$ denote the commutator of $T$ and $e$ divided by $\mu$, that is,
\begin{align*}
 T'(m) &= \frac{1}{\mu}(-eT + Te)(m) \\
       &= \frac{1}{\mu}(-ET(m) + T(Em)).
\end{align*}
Of course $T'$ depends on the choice of $E$ and $\mu$, and not only $T$. The derivative-like notation $T'$ is motivated by the formal properties of $T'$:
\begin{prop}
The operation $T\mapsto T'$ on $\hom_{\Mod(R)}(M,M)$ satisfies the conditions
\begin{align*}
 (rS)' &= r(S') \mbox{\ \ \ for\ all\ } r\in R,\\
 (ST)' &= S'T + ST',\mbox{\ \ \ and} \\
 (S+T)' &= S' + T'.
\end{align*}
\end{prop}
\begin{proof}
Elementary exercise.
\end{proof}
In other words, the operator $T\mapsto T'$, as an element in the $R$-algebra $\mathcal{O}$ of $R$-linear functions $A\rightarrow\hom_R(M,M)$, is an $R$-linear derivation. The function $T'$ vanishes if and only if $T$ commutes with multiplication by $E$; pursuing the analogy with differential calculus, it is as though we are considering a peculiar kind of calculus in which the ``constant functions'' are those which commute with multiplication by $E$. There is a more subtle difference between our situation and ordinary calculus, which makes the analogy with calculus somewhat strained: since the multiplication on $A$ is composition of the operations on $M$ (i.e., if $S,T\in A$, then $(ST)m = S(T(m))$), the rule $(ST)' = S'T + ST'$ is as though we have a product rule, rather than a chain rule, for function composition. 

It is not necessarily the case that iterated derivatives like $T \mapsto T''$ are defined, since $S'$ is only defined if $SE - ES$ is a multiple of $\mu$, and generally in the case $S = T'$ one cannot expect $T'E - ET' = \frac{1}{\mu} \left( SEE - 2ESE + EES\right)$ to be a multiple of $\mu$. 
However, all that is needed is to multiply by an appropriate power of $\mu$: if $T'$ is defined (i.e., if $TE - ET$ is a multiple of $\mu$), then $\mu^{n-1} T^{(n)}$ is defined for all $n\geq 1$. 
\begin{definition}\label{def of Dn}
Let $\delta\in\mathcal{O}$ denote the derivative $\delta T = T^{\prime}$ defined above.
For each nonnegative integer $j$ and each positive integer $n$, let $e^j \mu^{n-1}\delta^n\in \mathcal{O}$ denote $E^j\mu^{-1}$ times the $n$th iterate $\mu \delta\circ \dots \circ\mu \delta$ of $\mu\delta$.
Let 
\begin{align}
\label{eq 0493085} D_n := \sum_{i=1}^n \binom{n}{i} e^{n-i} \mu^{i-1} \delta^i\in \mathcal{O}.\end{align}
\end{definition}

Clearly $D_1T = \delta T = T'$ measures the failure of $T$ to commute with multiplication by $E$, but it is less obvious what $D_nT$ measures, for $n>1$. Happily, there is a simple answer to that question:
\begin{prop}\label{interpretation of Dn}
Let $n$ be a positive integer. Then $D_nT = 0$ if and only if $T$ commutes with multiplication by $E^n$.
\end{prop}
\begin{proof}
A straightforward induction proves that
\begin{align*}
 (\mu\delta)^nT &= \sum_{i=0}^n (-1)^i\binom{n}{i} E^i TE^{n-i}
\end{align*}
for all $n\geq 0$. 
Consequently we have
\begin{align}
\nonumber D_nT 
  &= \sum_{j=1}^n \binom{n}{j} E^{n-j}\mu^{-1}\delta^j T \\
\nonumber  &= \sum_{j=1}^n \binom{n}{j} E^{n-j}\mu^{-1}\sum_{i=0}^j(-1)^i\binom{j}{i}E^iTE^{j-i} \\
\nonumber  &= \mu^{-1}\sum_{j=1}^n\sum_{i=0}^j (-1)^i\binom{n}{j}\binom{j}{i}E^{n-j+i}TE^{j-i} \\
\nonumber  &= \mu^{-1}\sum_{k=0}^n \left( \sum_{i=0,i+k>0}^{n-k} (-1)^i\binom{n}{i+k}\binom{i+k}{i}\right) E^{n-k}TE^k \\
\label{use of lemma 1}  &= \mu^{-1}\left( TE^n - E^nT\right),
\end{align}
where \eqref{use of lemma 1} is due to Lemma \ref{binomial coeffs lemma}.
Since $M$ is assumed to be $\mu$-torsion-free, $D_nT = 0$ if and only if $TE^n = E^nT$ in $\hom_R(M,M)$.
\end{proof}

\begin{observation}\label{derivation observation}
Certainly not every element of $\mathcal{O}$ is a derivation. In particular, the iterated derivatives $\mu^{n-1}\delta^n$ are, in general, not derivations (unless $n=1$). However, it follows from the equation $D_nT = \mu^{-1}Te^{n} - \mu^{-1}e^nT$ obtained in the proof of Proposition \ref{interpretation of Dn} that $D_n$ is an $R$-linear derivation.
\end{observation}

Another consequence of the equation $D_nT = \mu^{-1}Te^{n} - \mu^{-1}e^nT$ is further $p$-divisibility of $D_nT$ when $p$ divides $n$:

\begin{definition-proposition}\label{def of delta}
Suppose the element $E\in M$ is congruent to $1$ modulo $\mu$. 
We write $\nu_p$ for $p$-adic valuation.
Suppose that $\mu$ is a multiple (in $R$) of $p$. Then for each $m\in M$,
the element $D_nTm$ of $M$ is divisible by $p^{\nu_p(n)}$.
Consequently we write $\Delta_n$ for the (unique, since $p\mid \mu$ and $M$ is assumed $\mu$-torsion-free) element of $\hom_R(A,\hom_R(M,M))$ such that $p^{\nu_p(n)} \Delta_n = D_n$.
\end{definition-proposition} 
\begin{proof}
Since $E\equiv 1 \mod \mu$, we have that $E-1 = \mu x$ for some $x\in M$, 
and so $E^n = (1 + \mu x)^n =  \sum_{j=0}^n \binom{n}{j}\mu^j x^j$.
Since the $p$-adic valuation of $\mu$ is at least $1$, the unique term of least $p$-adic valuation in the sum
$E^n - 1 = \sum_{j=1}^n \binom{n}{j}\mu^j x^j$ is the $j=1$ term $n\mu x$, i.e., the $p$-adic valuation of $E^n - 1$ is equal to $\nu_p(n) + \nu_p(\mu) + \nu_p(x)$. In particular, $E^n \equiv 1$ modulo $p^{\nu_p(n) + \nu_p(\mu)}$.

Consequently, for each $m\in M$, we have that $E^nm \equiv m$ modulo $p^{\nu_p(n) + \nu_p(\mu)}$, and so $TE^nm \equiv Tm$ modulo $p^{\nu_p(n) + \nu_p(\mu)}$.
Meanwhile, $E^nTm \equiv Tm$ modulo $p^{\nu_p(n) + \nu_p(\mu)}$ as well,
so $\frac{1}{\mu}\left( TE^n - E^nT\right) = D_nT$ is divisible by $p^{\nu_p(n)}$.
\end{proof}
So, for example, when $T = T_{\ell}$ is a Hecke operator and $E = E_{p-1}$ is the weight $p-1$ Eisenstein series, let $p = \mu$, and then $\Delta_1T$ is the operator on modular forms given by $\Delta_1Tf = D_1Tf = \frac{1}{p}\left( T(E_{p-1}f) - E_{p-1}Tf\right))$, and in general $\Delta_nTf = D_nTf = \frac{1}{p}\left( T(E_{p-1}^nf) - E_{p-1}^nTf\right)$ as long as $p\nmid n$. If $n$ is a multiple of $p$ but not $p^2$, then we have $\Delta_nTf = \frac{1}{p}D_nTf = \frac{1}{p^2}\left( T(E_{p-1}^nf) - E_{p-1}^nTf\right)$, and so on.

\subsection{Hochschild cohomology and derived eigenforms}
\label{Hochschild cohomology and derived eigenforms}

Recall (Def. \ref{def: abstract hecke algebra}) that the abstract Hecke algebra $\mathbf{A}$ is the polynomial algebra on the set of $T_p$, as $p$ runs through the set $\Pi$ of all primes. For the purposes of this article it will be useful to slightly generalize this notion, as follows:

\begin{definition}
\label{def:abstract Hecke for P}
Let $P\subseteq \Pi$ be a subset of the set of all prime numbers. 
By the {\em abstract Hecke algebra for $P$} we will mean the polynomial $\mathbb{Z}[P^{-1}]$-algebra on the set of generators $\{ T_p: p\in P\}$. We adopt the notation $\mathbf{A}_P$ for this algebra.
\end{definition}

This abstract Hecke algebra acts on the graded ring $M_*\otimes_{\ZZ} \ZZ[P^{-1}]$ as before. Note that if $P = \Pi$ is taken to be the set of all primes, then $\mathbf{A}_{\Pi} = \mathbf{A}\otimes_{\ZZ}\QQ$, where $\mathbf{A}$ is the abstract Hecke algebra of Def. \ref{def: abstract hecke algebra}. 


We will need to consider the Hochschild cohomology of abstract Hecke algebras. 
We refer to \cite{MR1269324} for basics of Hochschild cohomology. Most immediately relevant are $HH^0$ and $HH^1$: given a commutative $R$-algebra $A$, and an $A$-bimodule $M$, the degree $0$ Hochschild cohomology $HH^0(A; M)$ is isomorphic to the $R$-module of elements $m\in M$ such that $am = ma$ for all $a\in A$. Meanwhile, the degree $1$ Hochschild cohomology $HH^1(A; M)$ is isomorphic to the $R$-module of $1$-cochains modulo $1$-coboundaries. Here the $1$-cochains are the $R$-linear functions $\phi: A\rightarrow M$ such that $a_0\phi(a_1) - \phi(a_0a_1) + \phi(a_0)a_1$ for all $a_0,a_1\in A$, while the $1$-coboundaries are the $1$-cocycles of the form $a \mapsto am - ma$ for some fixed element $m\in M$. We will only consider bimodules of a particular form:

\begin{definition}\label{bimodules definition 1}
Suppose $A$ is a commutative $R$-algebra, and $M$ is a left $A$-module.
Suppose we are given an algebraic field extension $K/\mathbb{Q}$ and a $R$-algebra morphism $\lambda: A \rightarrow \mathcal{O}_K\otimes_{\mathbb{Z}} R$.
 Equip $\mathcal{O}_K\otimes_{\mathbb{Z}} M$ with the structure of an $A$-bimodule by 
\begin{align*}
 a\cdot m &= am \mbox{,\ the\ assumed\ left\ $A$-action\ on\ $M$,\ and} \\
 m\cdot a &= \lambda(a)m.
\end{align*}
We will write $M^{\lambda}$ for $\mathcal{O}_K \otimes_{\mathbb{Z}} M$ regarded as an $A$-bimodule in this particular way.
\end{definition}

Of course the motivating example of the above is the situation where $R = \mathbb{Z}[P^{-1}]$ for some subset of primes $P\subseteq \Pi$, $A$ is the abstract Hecke algebra $\mathbf{A}_P$, and $M = M_*\otimes_{\ZZ} \ZZ[P^{-1}]$ is the ring of holomorphic level one modular forms over $\mathbb{Z}[P^{-1}]$. In this case it is customary to call $\lambda$ an {\em eigencharacter} of $\mathbf{A}_P$:
\begin{definition}
An {\em eigencharacter} of the abstract Hecke algebra $\mathbf{A}_P$ is a ring homomorphism from $\mathbf{A}_P$ to a commutative ring.
\end{definition} 
If $P$ is a cofinite set of primes, then $HH^0(\mathbf{A}_P; M^{\lambda})$ recovers the $\mathcal{O}_K[P^{-1}]$-module of classical Hecke eigenforms with eigencharacter $\lambda$. Consequently, we regard the elements of the Hochschild groups $HH^n(\mathbf{A}_P; M^{\lambda})$, for $n>0$, as ``derived eigenforms'':

\begin{definition}
\label{def:derivedEigenform}
Let $P\subseteq \Pi$, let $\lambda: \mathbf{A}_P \rightarrow \mathcal{O}_K\otimes_{\mathbb{Z}} \mathbb{Z}[P^{-1}]$ be an eigencharacter and let $M = M_*\otimes_{\ZZ} \ZZ[P^{-1}]$ be the ring of holomorphic level one modular forms over $\mathbb{Z}[P^{-1}]$. A {\em derived eigenform} is an element of $HH^n(\mathbf{A}_P ; M^{\lambda})$, for some $n>0$.
\end{definition}

One of the main goals of this article is to demonstrate that derived eigenforms occur in two natural ways: in Corollary \ref{main thm cor 123409} below, we describe `topological' Hecke eigenforms over $2$-cell complexes in terms of derived Hecke eigenforms; and below, in Definition-Proposition \ref{def of dot-cup product}, we construct degree-raising operations on Hochschild cohomology which naturally take (classical) Hecke eigenforms to derived Hecke eigenforms. In Corollary \ref{main thm cor 123409}, we see that derived Hecke eigenforms which describe topological Hecke eigenforms indeed arise from these degree-raising operations, in particular from the dot-cup product with certain fixed classes in Hochschild cohomology.

We now give some general algebraic constructions that will lead to explicit examples of derived Hecke eigenforms. 
First, in Definition \ref{bimodules definition 2} we recall the standard $A$-bimodule structure on $\hom_R(M,N)$:
\begin{definition}\label{bimodules definition 2}
Given left $A$-modules $M$ and $N$, we equip $\hom_R(M,N)$ with the structure of an $A$-bimodule by
letting, for any $a\in A$ and $f\in \hom_R(M,N)$, 
\begin{align*}
 (a \cdot f)(m) &= af(m) \mbox{,\ and} \\
 (f \cdot a)(m) &= f(am).
\end{align*}
In other words: the left action of $A$ on $\hom_R(M,N)$ is by post-composition, while the right action of $A$ on $\hom_R(M,N)$ is by pre-composition. Consequently, \[ HH^0(A; \hom_R(M,N))\cong \hom_A(M,N).\]
\end{definition}

Definition-Proposition \ref{def of dot-cup product} refers to the cup product on Hochschild cochains. It is defined as follows: given an $R$-algebra $A$ and $A$-bimodules $M,N$, an $i$-cochain $f: A^{\otimes_R i}\rightarrow M$, and a $j$-cochain $g: A^{\otimes_R j}\rightarrow N$, the cup product $f\cup g$ is the $(i+j)$-cochain $A^{\otimes_R (i+j)}\rightarrow M\otimes_A N$ given by $(f\cup g)(a_1\otimes \dots \otimes a_{i+j}) = f(a_1\otimes \dots \otimes a_i)\otimes g(a_{i+1}\otimes \dots \otimes a_{i+j})$. (As far as the authors know, \cite{MR0161898} was where the cup product was first defined for a general associative algebra, although a version for group algebras was considered earlier in \cite{MR0019092}. A nice exposition can be found in the unpublished notes \cite{drupieskinotes}.)

\begin{definition-proposition}\label{def of dot-cup product}
Given $R,A,M,N,\lambda$ as in Definitions \ref{bimodules definition 1} and \ref{bimodules definition 2}, there exists a unique $A$-bimodule map
$\chi: \hom_R(M,N)\otimes_A M^{\lambda} \rightarrow N^{\lambda}$
such that the composite of $\chi$ with the natural quotient map $\hom_R(M,N)\otimes_R M^{\lambda}\rightarrow \hom_R(M,N)\otimes_A M^{\lambda}$
is the evaluation map $\ev: \hom_R(M,N)\otimes_R M\rightarrow N (= N^{\lambda})$ sending $f\otimes m$ to $f(m)$. 

Consequently the composite of the cup product 
\[ HH^*(A; \hom_R(M,N)) \otimes_R HH^*(A; M^{\lambda})\stackrel{\cup}{\longrightarrow} HH^*(A; \hom_R(M,N)\otimes_A M^{\lambda})\]
with the map $HH^*(A; \hom_R(M,N)\otimes_A M^{\lambda}) \stackrel{HH^*(A; \chi)}{\longrightarrow} HH^*(A; N^{\lambda})$ is a $R$-linear map
\begin{equation}\label{dotcup} HH^*(A; \hom_R(M,N)) \otimes_R HH^*(A; M^{\lambda}) \rightarrow HH^*(A; N^{\lambda}).\end{equation}
We call the resulting $R$-bilinear product map $HH^*(A; \hom_R(M,N)) \times HH^*(A; M^{\lambda}) \rightarrow HH^*(A; N^{\lambda})$ the {\em dot-cup product}, written $\cupdot$.
\end{definition-proposition}
\begin{proof}
Let $q$ denote the quotient map $\hom_R(M,N)\otimes_R M^{\lambda}\rightarrow \hom_R(M,N)\otimes_A M^{\lambda}$.
If $\ev$ factors through $q$, then it does so uniquely, since $q$ is surjective. So we need to check that $\ev$ vanishes on the kernel of $q$, i.e., that
$\ev(fa\otimes m - f\otimes am)$ vanishes for all $a\in A$ and $m\in M$ and $f\in\hom_R(M,N)$.
This is routine and follows immediately from the definitions.

We need to check that the resulting factor map $\chi$ is an $A$-bimodule map. This is also routine:
\begin{align*}
 a\chi(f\otimes m)
  &= af(m)
  = \chi(af\otimes m), \mbox{\ \ \ and} \\
 (\chi(f\otimes m))a 
  &= \lambda(a)\chi(f\otimes m)
  = \chi(f\otimes \lambda(a)m)
  = \chi((f\otimes m)a).
\end{align*}

That the cup product is $R$-bilinear is proven in \cite{MR0161898}. Since $\cupdot$ is the composite of $\cup$ with the $R$-linear map $HH^*(A; \chi)$, the composite \eqref{dotcup} is also $R$-linear, i.e., $\cupdot$ is $R$-bilinear.
\end{proof}

Given a Hecke eigenform $f\in HH^0(\mathbf{A}_P; M^{\lambda})$ and an element $g$ of \linebreak $HH^0(\mathbf{A}_P; \hom_R(M,N))$, their dot-cup product $f\cupdot g\in HH^0(\mathbf{A}_P; N^{\lambda})$ is simply the image of $f$ under the map $\hom_{\mathbf{A}_P}(M,N)$ associated to $g\in HH^0(\mathbf{A}_P; \hom_R(M,N))$ $\cong \hom_{\mathbf{A}_P}(M,N)$. As we show below in Proposition \ref{prop:HH0}, this construction is not very interesting for modular forms.
On the other hand, given $f\in HH^0(\mathbf{A}_P; M^{\lambda})$ and an element $\kappa\in HH^1(\mathbf{A}_P; \hom_R(M,N))$ of degree one, the dot-cup product $f\cupdot \kappa\in HH^1(\mathbf{A}_P; N^{\lambda})$ is a derived Hecke eigenform, in the sense of Definition \ref{def:derivedEigenform}. Whenever the space $HH^1(\mathbf{A}_P; \hom_R(M,N))$ is non-zero, dot-cupping with elements of $HH^1(\mathbf{A}_P; \hom_R(M,N))$ then yields an ample supply of algebraic operations which take as input a classical Hecke eigenform and produce as output a derived Hecke eigenform. We are thus led to establish some structural and vanishing results for the $HH^0$- and $HH^1$- terms:

\begin{prop}
\label{prop:HH0}
Let $P$ be a set of primes which is cofinite, i.e., $P$ contains all but finitely many prime numbers. Let $k,k^{\prime}$ be integers with $k\neq k^{\prime}$, and let $M_k$ (respectively, $M_{k^{\prime}}$) be the $\mathbb{Z}[P^{-1}]$-module of holomorphic modular forms over $\mathbb{Z}[P^{-1}]$ of level $1$ and weight $k$ (respectively, weight $k^{\prime}$).
Then $HH^0(\mathbf{A}_P; \hom_{\mathbb{Z}[P^{-1}]}(M_k,M_{k^{\prime}}))$ vanishes.
\end{prop}
\begin{proof}
Let $\gamma \in HH^0(\mathbf{A}_P; \hom_{\mathbb{Z}[P^{-1}]}(M_k,M_{k'})) = \hom_{\mathbf{A}_P}(M_k,M_{k'})$. We want to show that $\gamma$ is the zero map. By extending scalars to $\overline{\QQ}$, we will do so by applying $\gamma$ to (normalized) Hecke eigenforms in $M_k\otimes \overline{\QQ}$. These come in two types:
\begin{itemize} 
\item[(1)] Cuspidal eigenforms $f$. In this case the Hecke eigenvalue of $T_{\ell}$ is an algebraic integer $a_{\ell}$ lying in a totally real number field $K$. There is a corresponding 2-dimensional irreducible Galois representation 
$$
\rho_{f}: \Gal(\overline{\QQ}/\QQ) \rightarrow \GL_2(K\otimes \mathbb{Q}_p)
$$
satisfying $\Tr\rho_{f}(\Frob_{\ell}) = a_{\ell}$ and $\det \rho_{f}(\Frob_{\ell})=\ell^{k-1}$, for all $\ell \neq p$. 

\item[(2)] Scalar multiples of the Eisenstein series $E_k$. In this case the Hecke eigenvalue of $T_{\ell}$ is $\sigma_{k-1}(\ell) = 1 + \ell^{k-1}$. The corresponding 2-dimensional Galois representation 
$$
\rho_{E_k}: \Gal(\overline{\QQ}/\QQ) \rightarrow \GL_2(\mathbb{Q}_p)
$$
decomposes as $1\oplus \chi^{k-1}_p$, where $\chi_p: \Gal(\overline{\QQ}/\QQ)  \rightarrow \ZZ_p^{\times}$ is the $p$-adic cyclotomic character, satisfying $\chi_p(\Frob_{\ell}) = \ell$, for all $\ell \neq p$. 

\end{itemize}

As is well-known, the space $M_k\otimes \overline{\QQ}$ has a basis of eigenforms $\{g_1, \ldots, g_{d_k}\}$. We will show that $\gamma$ is the zero map in this basis. Indeed, if $g$ is an eigenform then 
$$
\gamma (T_{\ell}g) = a_{\ell} \gamma(g) = T_{\ell}\gamma(g)
$$
for all Hecke operators $T_{\ell}$ with $\ell \in P$. This means that $h := \gamma(g) $ is an eigenform of weight $k'$ with the same eigencharacter $\lambda(T_{\ell}) = a_{\ell}$ as $g$, which is of weight $k$. This is impossible. Indeed let $K$ be the number field containing the Fourier coefficients of $h$ (which contains those of $g$ as well) and let 
$$
\rho_{h}, \rho_g : \Gal(\overline{\QQ}/\QQ) \rightarrow \GL_2(K\otimes \QQ_p)
$$
be the 2-dimensional $p$-adic Galois representations attached to $h,g$, respectively. If the eigencharacters of $h$ and $g$ are the same, then 
$$
\Tr \rho_g (\Frob_{\ell}) = \Tr \rho_h (\Frob_{\ell}), \quad \ell \in P-\{p\}.
$$
Since $P$ is cofinite, by the Chebotarev density theorem (\cite{Serre}, II.2.2) this means that the characters of the entire Galois representations are the same. Therefore 
$$
\ell^{k-1} = \det\rho_{g}(\Frob_{\ell}) = \det\rho_{h}(\Frob_{\ell}) = \ell^{k'-1}
$$
for all primes $\ell$, since for two-by-two matrices we have the identity $\det M  = (\Tr(M)^2 - \Tr(M^2))/2$. This implies that $k -1 = k'-1$, which is impossible since $k \neq k'$.
\end{proof}
\begin{corollary}\label{HH0 splitting}
Let $P$ be a set of primes which is cofinite, i.e., $P$ contains all but finitely many prime numbers. Let $M_*$ be the graded $\mathbb{Z}[P^{-1}]$-algebra of modular forms over $\mathbb{Z}[P^{-1}]$.
Then $HH^0(\mathbf{A}_P; \hom_{\mathbb{Z}[P^{-1}]}(M_*,M_*))$ splits as a direct sum
\[ HH^0(\mathbf{A}_P; \hom_{\mathbb{Z}[P^{-1}]}(M_*,M_*)) \cong 
 \coprod_{k\in\mathbb{Z}} HH^0(\mathbf{A}_P; \hom_{\mathbb{Z}[P^{-1}]}(M_k,M_k)).\]
\end{corollary}

Corollary \ref{HH0 splitting} gives us some understanding of $HH^0(\mathbf{A}_P; \hom_{\mathbb{Z}[P^{-1}]}(M_*,M_*))$; in particular, it tells us that every nonzero element of $HH^0(\mathbf{A}_P; \hom_{\mathbb{Z}[P^{-1}]}(M_*,M_*))$ is a {\em weight-preserving} operation on modular forms.

Now we focus on gaining an understanding of $HH^1(\mathbf{A}_P; \hom_{\mathbb{Z}[P^{-1}]}(M_*,M_*))$. We begin abstractly, without the assumption that $M$ is a ring of modular forms or that $A$ is a Hecke algebra:
\begin{definition-proposition}\label{def of kappas}
Suppose $A$ is a commutative $R$-algebra, and $M$ is a left $A$-module equipped with the structure of a commutative $R$-algebra.
Suppose we are given an element $\mu\in R$ and an element $E\in M$ such that $M$ is $\mu$-torsion-free, and such that $TE \equiv ET$ modulo $\mu$ for all $T\in A$.
Then, for each positive integer $n$ and each prime number $p$ which divides $\mu$, we have a Hochschild $1$-cocycle $\phi^E_n: A \rightarrow \hom_R(M,M)$ given by $\phi^E_n(T) = \Delta_nT$, as defined in Definition-Proposition \ref{def of delta}. 
Let $\kappa^E_n\in HH^1(A; \hom_R(M,M))$ denote the cohomology class of $\phi^E_n$. 
\end{definition-proposition}
\begin{proof}
The claim that $\phi^E_n$ is a Hochschild $1$-cocycle is equivalent to the claim that $\Delta_n$ is an $R$-linear derivation, which follows from Observation \ref{derivation observation}.
\end{proof}
From Proposition \ref{interpretation of Dn}, we know that the condition that $\phi^E_n = 0$ is equivalent to the condition that the action of $T$ on $M$ commutes with multiplication by $E$ for all $T\in A$. 
In the case where $A$ is the abstract Hecke algebra over $\mathbb{Z}[P^{-1}]$ and $M = M_*$ is the ring of holomorphic modular forms, the vanishing of $\phi^E_n$ is equivalent to the action of every Hecke operator $T_{\ell}$, $\ell \in P$, on $M_*$ commuting with multiplication by the $n$th power of the modular form. Of course this condition almost always (depending on the choices of $P$ and $E$ and $n$) fails. 

Rather than asking for $\phi^E_n$ to be equal to zero, an apparently weaker condition is to ask that the Hochschild $1$-cocycle $\phi^E_n$ be merely cohomologous to zero, i.e, to ask that $\kappa^E_n = 0$. 
This is equivalent to the existence of some $R$-linear map $f: M_* \rightarrow M_*$ such that $Tf - fT = \phi^E_n$. Recall, from Definition-Proposition \ref{def of delta}, that $\mu\in R$ has the property that $\mu p^{\nu_p(n)}\phi^E_n$ is indeed a coboundary, since $\phi^E_n = \frac{1}{\mu p^{\nu_p(n)}} (Te^n - e^nT)$. Consequently $\mu p^{\nu_p(n)}\kappa^E_n = 0\in HH^1(A; \hom_R(M_*,M_*))$. 

We now specialize to the situation where $A = \mathbf{A}_P$ is the abstract Hecke algebra over $\mathbb{Z}[P^{-1}]$, $E$ is the Eisenstein series $E_{p-1}$, $\mu=p$, and $M = M_*$ is the ring of holomorphic modular forms of level $1$.
\begin{theorem}\label{nontriviality of all kappas}
Let $P$ be a cofinite set of primes. Then for each prime $p\notin P$ and each positive integer $n$,
the class of $\kappa^{E_{p-1}}_n$ is non-trivial and of order $p^{1+\nu_p(n)}$ in $HH^1\left(\mathbf{A}_P,\hom_{\mathbb{Z}[P^{-1}]}(M_*,M_*)\right)$.
\end{theorem}
\begin{proof}
In this case $\mu=p$, and so the triviality of the cohomology class $p^{1+\nu_p(n)}\kappa^{E_{p-1}}_n$ has already been established. 
To show that the $p^{\nu_p(n)}\kappa^{E_{p-1}}_n$ is non-trivial, suppose there exists a $\mathbb{Z}[P^{-1}]$-linear map $f:M_* \rightarrow M_*$ such that 
$\frac{1}{p}\left(T E^n_{p-1}  - E^n_{p-1}T \right) = Tf - fT.$
Rearranging terms, we get 
$(\frac{E^n_{p-1}}{p} - f)T = T(\frac{E^n_{p-1}}{p} - f)$, so that the map $\frac{E^n_{p-1}}{p} - f$ is $\mathbf{A}_P$-equivariant and must send $M_k$ to $M_{k+n(p-1)}$. By Proposition \ref{prop:HH0} and the fact that inverting a prime commutes with Hochschild cohomology, this means that 
$f = \frac{E^n_{p-1}}{p}$,
that is, $f: M_* \rightarrow M_*$ must be the map given by $g \mapsto \frac{1}{p}E^n_{p-1}g$. But this map does not land in $M_{k'}$, since $p$ does not divide $E^n_{p-1} = 1 + \ldots $, so no such homomorphism $f$ exists.
\end{proof}

\begin{remark}
As a consequence of Theorem \ref{nontriviality of all kappas}, we have summands in \linebreak $HH^1(\mathbf{A}_P; \hom_R(M_*,M_*))$ which exhibit the same familiar Kummer congruence pattern as the denominators of the special values of $\zeta(s)$ at negative integers, the Adams-Novikov $1$-line, the first flat cohomology $H^1_{fl}(\mathcal{M}_{fg}; \omega^*)$ of the moduli stack of one-dimensional formal groups, the image of the Whitehead $J$-homomorphism in the stable homotopy groups of spheres, and so on.
It seems plausible that the summands of $HH^1(\mathbf{A}_P; \hom_R(M_*,M_*))$ generated by the cohomology classes $\kappa^{E_{p-1}}_n$ for various $n$ and $p\notin P$ exhaust all of $HH^1(\mathbf{A}_P; \hom_R(M_*,M_*))$, but we have not tried to prove that. We also do not make any attempt to compute $HH^n(\mathbf{A}_P; \hom_R(M_*,M_*))$ for $n>1$, and we do not know if there are summands in $HH^n(\mathbf{A}_P; \hom_R(M_*,M_*))$ for $n>1$ whose orders are number-theoretically or topologically meaningful.
\end{remark}

\subsection{Computation of derived Hecke eigenforms in degree one}

Next, we provide a method to produce explicit derived eigenforms in $HH^1$. Recall that, given a commutative ring $R$ and an $R$-algebra $A$, we write $A^{\epsilon}$ for the {\em enveloping $R$-algebra of $A$}, that is, $A^{\epsilon} = A\otimes_R A^{\op}$. We have a natural equivalence between left $A^{\epsilon}$-modules and $A$-bimodules, and $HH^n(A,M) \cong \Ext^n_{A^{\epsilon}/R}(A,M)$, where $\Ext^n_{A^{\epsilon}/R}$ is $\Ext$ relative to the allowable class consisting of the $R$-split sequences of $A^{\epsilon}$-modules; see e.g. Lemma 9.1.3 of \cite{MR1269324}.
\begin{prop}
Let $R$ be a commutative ring, let $S$ be a set, and let $A$ be the symmetric $R$-algebra on $S$, i.e., $A$ is the polynomial $R$-algebra on the set of generators $S$.
Given an $A$-bimodule $M$, we have isomorphisms of $A$-modules
\begin{align}
\label{iso 340958}  
 HH^1(A; M)
   &\cong \Ext^1_{A^{\epsilon}}(A,M) \\
\label{iso 340959}   &\cong M^{\prime}/M^{\prime\prime},
\end{align}
where $M^{\prime}$ is the submodule of $\prod_{s_0,s_1} M \{ e_{s_0,s_1}\}$ consisting of elements $\prod_{s_0,s_1\in S} m_{s_0,s_1}e_{s_0,s_1}$ satisfying 
\[ s_0m_{s_1} - m_{s_1}s_0 - s_1m_{s_0} + m_{s_0}s_1 =0\]
for all $s_0,s_1$, and $M^{\prime\prime}$ is the submodule of $M^{\prime}$ consisting of the 
elements of $\prod_{s_0,s_1} M \{ e_{s_0,s_1}\}$ of the form $\prod_{s\in S} (sm-ms)e_s$ for some $m\in M$. 
\end{prop}
\begin{proof}
Isomorphism \eqref{iso 340958} is an easy consequence of freeness of $A$ as an $R$-module; see e.g. Corollary 9.1.5 of \cite{MR1269324}. The rest of this proof will be concerned with isomorphism \eqref{iso 340959}.
Let $D^{\bullet}$ be the differential graded $R[x]^{\epsilon}$-algebra $(R[x]^{\epsilon})[y]/y^2$ with $\left| y\right| = 1$, and with differential $d(y) = x\otimes 1 - 1\otimes x$, Then the homology of $D^{\bullet}$ is $R[x]$ concentrated in grading degree $0$, i.e., $D^{\bullet}$ is a free $R[x]^{\epsilon}$-module resolution of $R[x]$. Consequently, given a set $S$,
if we write $\#(S)$ for the cardinality of $S$, and  
if we write $A$ for the $\#(S)$-fold tensor power $R[x]^{\otimes_R \#(S)}$,
then the $\#(S)$-fold tensor power $(D^{\bullet})^{\otimes_R \#(S)}$ of $D^{\bullet}$ is a DGA which is a free $A^{\epsilon}$-module resolution of $A$.
This resolution, in low degrees, is:
\begin{equation}\label{resolution 34091} 0 \leftarrow A^{\epsilon} \stackrel{d_{0}}{\longleftarrow} \coprod_{s\in S} A^{\epsilon}\{ e_s\} \stackrel{d_1}{\longleftarrow} \coprod_{s_0,s_1\in S} A^{\epsilon}\{ e_{s_0,s_1}\} \stackrel{d_{2}}{\longleftarrow} \dots \end{equation}
with differentials
\begin{align*}
 d_0(e_s) &= s\otimes 1 - 1\otimes s \\
 d_1(e_{s_0,s_1}) &= (s_0\otimes 1 - 1\otimes s_0) e_{s_1} - (s_1\otimes 1 - 1\otimes s_1) e_{s_0} .
\end{align*}
Now given an $A$-bimodule $M$, we regard $M$ as an $A^{\epsilon}$-module, and we apply $\hom_{A^{\epsilon}}(-, M)$ to \eqref{resolution 34091} to get the cochain complex
\[
 0 \rightarrow M \stackrel{d^0}{\longrightarrow} \prod_{s\in S} M\{ e_s\} \stackrel{d^1}{\longleftarrow} \prod_{s_0,s_1\in S} M\{ e_{s_0,s_1}\} \]
with 
\begin{align*}
 d^0(m) &= \sum_{s\in S} (sm-ms)e_s \\
 d^1(\sum_{s\in S}m_se_s) &= \sum_{s_0,s_1\in S} (s_0m_{s_1} - m_{s_1}s_0 - s_1m_{s_0} + m_{s_0}s_1 )e_{s_0,s_1}
\end{align*}
So we have an isomorphism between $\Ext^1_{A^{\epsilon}}(A, M)$ and $\ker d^1/\im d^0$. 
\end{proof}

Thanks to this result, we can give an explicit description of derived eigenforms belonging to $HH^1$, as follows. Let $P \subseteq \Pi$ be a set of primes, let $\mathbf{A}_P$ be the $\ZZ[P^{-1}]$-algebra generated by the Hecke operators $T_p$, $p\in P$, as above. Let $R$ be a commutative ring, let $\lambda: \mathbf{A}_P\otimes_{\ZZ} R\rightarrow R$ be an eigencharacter and let $M = (M_k\otimes_{\mathbb{Z}} R)^{\lambda}$ as above. Then the derived eigenforms in $HH^1(\mathbf{A}_P\otimes_{\ZZ} R; (M_k\otimes_{\mathbb{Z}} R)^{\lambda})$ consist of the submodule $M'$ of formal sums of pairs $\sum_{p,\ell \in P}(f_p, f_{\ell}) \in \prod_{p,\ell \in P} (M_k\otimes R)^2$ such that 
\begin{equation}
\label{eqn:cochainCondition}
T_p\,f_{\ell} - \lambda(T_p)\,f_{\ell} = T_{\ell}\,f_p - \lambda(T_{\ell})\,f_p
\end{equation}
modulo the submodule $M''$ of ``principal'' elements consisting of the formal sums
$$
\sum_{p,\ell\in P}(f_p, f_{\ell}) = (T_p\,f -\lambda(T_p)\,f, T_{\ell}\,f - \lambda(T_{\ell})\,f)
$$ 
corresponding to some fixed $f\in M_k\otimes R$. 

For example, when $\mathrm{rk}_{\ZZ} M_k = 1$ (so that $k=0,4,6,8,10,14$) and the eigencharacter $\lambda:\mathbf{A}_{\Pi} \rightarrow \QQ$  factors through the weight $k$ Hecke algebra $\mathbf{A}_k\otimes_{\mathbb{Z}}\QQ$ (i.e. it is an Eisenstein series eigencharacter given by $\lambda(T_p) = 1 + p^{k-1}$), then every element $f \in M_k\otimes_{\mathbb{Z}}\QQ$ is a $\lambda$-eigenform; therefore condition \eqref{eqn:cochainCondition} is always satisfied, so that  $M' \simeq \prod_{p,\ell \in \Pi} (M_k\otimes_{\mathbb{Z}}\QQ)^2$ and $M'' = 0$.

Next, fix a prime $p>3$ and let $P = \Pi - \{p\}$, so that instead of inverting every prime we invert all but one prime. The following example shows that in this case it is possible for $HH^1$ to be non-vanishing, i.e. that there are non-trivial derived Hecke eigenforms in degree one. 

\begin{example}\label{derived eigenform example} 

Let 
$$
\Delta = q\prod_{n=1}^{\infty} (1 - q^n)^{24} = \sum_{n=1}^{\infty} \tau(n)q^n,\quad\quad  q = e^{2\pi i z}, \Im[z]>0
$$
be the unique normalized cuspidal eigenform in $S_{12}$, with associated eigencharacter $\tau$, taking values in $\ZZ$. Let $p \geq 5$ be prime, let $P = \Pi - \{p\}$ and consider $\tau$ as an eigencharacter
$$\tau: \mathbf{A}_P \rightarrow \ZZ[P^{-1}] = \ZZ_{(p)}.$$
Let $k = 12 + p -1 = 11 + p$ and as usual denote by $T_{k,\ell}$ the classical (level one) weight $k$, $\ell$-Hecke operator. The operator $T_{k,\ell} - \tau(\ell)$ is integral on $M_k\otimes \ZZ_{(p)}$, since $T_{k,\ell}$ is integral and $\tau(\ell)\in \ZZ$. Therefore $\det(T_{k,\ell} - \tau(\ell)) \in \ZZ$. The operator $T_{k,\ell} - \tau(\ell)$ is invertible on $M_{k}\otimes_{\mathbb{Z}} \ZZ_{(p)}$ if and only if $p\nmid \det(T_{k,\ell} - \tau(\ell))$. This is never the case, that is, 
$$
p \mid \det(T_{k,\ell} - \tau(\ell)) \quad \forall \ell \in P
$$
Indeed, consider the modular form $f = \Delta\cdot E_{p-1} \in M_{k}\otimes_{\mathbb{Z}} \ZZ_{(p)}$. This is not an eigenform in $M_k$, as can be verified by direct computation. However, it is an eigenform for the mod $p$ Hecke algebra acting on $M_k\otimes_{\mathbb{Z}} \mathbb{F}_p$. Indeed, note that 
\begin{itemize}
\item [(1)] $T_{k,\ell} \equiv T_{12,\ell} \mod p$ for all primes $\ell$. For $\ell \neq p$, this follows from the explicit formula
$$
T_{k,\ell}(\sum a_n q^n) = \sum a_{n\ell} q^n + \ell^{k}\sum a_n q^{n\ell}
$$
recalling that $\ell^k \equiv \ell^{12}$ by Fermat's Little Theorem. For $p = \ell$, we have the congruence 
$$
T_{k,p}(\sum a_n q^n) \equiv \sum a_{np} q^n \equiv T_{12,p}(\sum a_n q^n) \mod p
$$
\item[(2)] $E_{p-1} \equiv 1 \mod p$. 
\end{itemize}
Therefore the operator $T_{k,\ell} - \tau(\ell)$ is never invertible mod $p$ and the above vanishing argument cannot be applied. Indeed, note that we know  from Theorem \ref{nontriviality of all kappas} that the group 
$$
HH^1\left(\mathbf{A}_P; \left( M_k\otimes_{\mathbb{Z}}\ZZ_{(p)}\right)^{\tau}\right)
$$
is non-zero, since it contains the non-trivial $p$-torsion element $\kappa_1^{E_{p-1}} \cupdot \Delta$, a derived eigenform in degree one. Explicitly, this class is represented by the derived eigenform whose components $f_{\ell}$, for every prime $\ell \in P$, are given by 
$$
f_{\ell} = \frac{1}{p}\left( T_{k,\ell}(\Delta\, E_{p-1}) - \tau(\ell)\Delta E_{p-1} \right).
$$
Note that $p f$ is indeed principal, since $pf_{\ell} =   T_{k,\ell}f - \tau(\ell)f $ with $f = \Delta\,E_{p-1} \in M_k \otimes \ZZ_{(p)}$. 
\end{example}

\section{Review of established ideas from modular forms and elliptic homology.}
\label{Review of established...}

\subsection{Review of some well-known level $1$ elliptic homology theories.}
Here is a definition from \cite{MR1037690} which is a slight refinement of that from \cite{MR970281}:
\begin{definition}\label{def of Ell}
Let $\Ell_*$ be the graded ring $\mathbb{Z}[\frac{1}{6}][g_2,g_3][\Delta^{-1}]$ of weakly holomorphic modular forms, with gradings
\begin{align*}
 \left| g_2\right| &= 8,\mbox{\ \ and} \\
 \left| g_3\right| &= 12,
\end{align*}
and where $\Delta$ is the discriminant of the Weierstrass cubic $Y^2 = 4X^3 - g_2X - g_3$, i.e., $\Delta = g_2^3 - 27g_3^2$. Equip $\Ell_*$ with the structure of a $MU_*$-algebra via the ring map $MU_*\stackrel{\cong}{\longrightarrow} L \rightarrow \Ell_*$ classifying the formal group law of the Weierstrass cubic $Y^2 = 4X^3 - g_2X - g_3$. The resulting graded $MU_*$-module $\Ell_*$ is Landweber exact, and by {\em weakly holomorphic elliptic homology} 
we mean the Landweber exact homology theory $\Ell_*(-) \cong MU_*(-)\otimes_{MU_*}\Ell_*$. 
\end{definition}
Put another way, $\Ell_*$ is the generalized homology theory represented by the spectrum of ``good reduction topological modular forms'' with $6$ inverted, $TMF[\frac{1}{6}]$.

In particular, since $\Delta$ is inverted in $\Ell_*$, for $p>3$ the $p$-localization of $\Ell_*$ is the ``good reduction topological modular forms'' generalized homology theory $(TMF_{(p)})_*$ whose global model is $TMF_*$, which is a coarser invariant than the ``semistable reduction topological modular forms'' homology theory whose global model is $Tmf_*$, and coarser still than the connective homology theory $tmf_*$. After inverting $6$, 
we have: 
\begin{definition} \label{def of elll-cohomology}
Let $\elll_*(-)$ be the generalized homology theory obtained by inverting $6$ in $\tmf_*$, so that $\elll_*(S^0)$ is the graded ring 
of holomorphic modular forms over $\mathbb{Z}[\frac{1}{6}]$, with gradings equal to those in Definition \ref{def of Ell}, i.e., the grading of a modular form is equal to $2$ times its weight.
We refer to $\elll_*$ as {\em holomorphic elliptic homology}.
\end{definition}

\begin{remark}\label{remark on why 6 is inverted 1}
It is well known that $tmf$ and $Tmf$ have special behavior at the primes $2$ and $3$, and in particular both of $\pi_*(tmf)$ and $\pi_*(Tmf)$ have $2$-torsion and $3$-torsion elements, hence cannot coincide with the ring of holomorphic modular forms unless $6$ is inverted. Throughout this paper, we almost always work with $6$-inverted $tmf$-homology, i.e., holomorphic elliptic homology. Here are two reasons why:
\begin{itemize}
\item
Baker gives a construction of a topological Hecke operator $\tilde{T}_n$ on a generalized homology theory in which $n$ is inverted. Baker's construction is described below in \cref{Baker's Hecke operators...}. This construction depends critically on the dual generalized cohomology theory being {\em complex-oriented}. 
The generalized cohomology theories $\tmf^*$ and $\Tmf^*$ do not become complex oriented until $6$ is inverted. In the literature one can find other constructions of Hecke operators as unstable operations, e.g. as a power operation in \cite{MR2219307} and \cite{MR1637129}, 
but these are not stable operations. We do not know any way to construct Hecke operators as stable cohomology operations on $\tmf$ or $\Tmf$ without inverting $6$, although we do not know any reason why this should be impossible.
\item
However, in Remark \ref{remark on why 6 is inverted 2}, we give a geometric argument why, even if Hecke operators as stable operations on $\tmf$ {\em can} be constructed without inverting $6$, the resulting Hecke operators must act on the $2$-torsion and $3$-torsion in $\tmf$-homology in a way which is determined by the action of $\mathbb{F}_p^{\times}$ on $\Gamma_1(p)$ by Diamond operators, and consequently is trivial for a positive-density set of primes $p$. Consequently the computation of an action of topological Hecke operators on $2$-torsion and $3$-torsion in $tmf$ winds up being trivial for many primes, and at primes where the Hecke action is nontrivial, its effect on $2$-cell complexes is algebraically approachable using tools similar to those used in the present paper. 
\end{itemize}
Our understanding from \cite{daviesinpreparation} is that J. M. Davies is preparing a paper with positive results on the construction of Hecke operators on $tmf$, and that Davies is considering the action of the Hecke operators on the $2$-torsion and $3$-torsion in $tmf$-homology. 
\end{remark}

\subsection{Adams operations and topological Hecke operators.}
\label{Baker's Hecke operators...}

\begin{definition}
Suppose that $E_*$ is a generalized cohomology theory, and suppose that $P\subseteq \Pi$ is a set of primes.
\begin{itemize}
\item 
We say that $E^*$ {\em has Adams operations in $P$} if $E^n(S^0)$ vanishes for all odd integers $n$, and if, for each prime $p\in P$, we have made a choice of stable natural transformation $\Psi^p: E^*(-) \rightarrow E^*(-)$ such that:
\begin{itemize}
\item each $\Psi^p$ preserves grading degrees, 
\item each $\Psi^p$ is multiplicative, i.e., $\Psi^p: E^*(X) \rightarrow E^*(X)$ is a ring endomorphism for every pointed space $X$, and
\item for each integer $k$, the operation $\Psi^p: E^{2k}(S^0) \rightarrow E^{2k}(S^0)$ agrees with\footnote{This requires, in particular, that ``multiplication by $p^{-k}$'' is defined on $E^{2k}(S^0)$. When this happens, it is usually because either $E^n(S^0)$ is trivial for $n>0$, or because $E^*(\pt)$ is a $\mathbb{Z}[\frac{1}{p}]$-module.} multiplication by $p^{-k}$. 
\end{itemize}
\item 
Suppose that $E^*$ has Adams operations in $P$. We say that $E^*$ {\em has Hecke operators in $P$} if, for each $n \in \mathbb{N}$ such that $1/n \in \ZZ[P^{-1}]$, we have made a choice of degree-preserving stable natural transformation 
$\tilde{T}_n: E^*(-) \rightarrow E^*(-)$ 
such that 
\begin{align}
\label{hecke relation 0} \tilde{T}_{mn} &= \tilde{T}_m\circ \tilde{T}_n, \quad (m,n) = 1\ \mbox{\ \ \ and} \\
\label{hecke relation} \tilde{T}_{p^{r+2}}(x) &= \tilde{T}_p\left(\tilde{T}_{p^{r+1}}(x)\right) - \frac{1}{p}\Psi^p\tilde{T}_{p^r}(x) \quad p\in P,\ r\geq 0,\ x\in E^*(X).\end{align}
\end{itemize}
\end{definition}


\begin{observation}\label{brown rep observation}
If $E^*$ has Hecke operators in $P$, then Brown representability allows us to represent the actions of $\Psi^p$ and of $\tilde{T}_p$ on $E^*$ by (defined only up to homotopy) maps $\mathbb{\Psi}^p: E \rightarrow E$ and $\mathbb{T}_p: E\rightarrow E$ of the representing spectrum $E$ of $E$-cohomology. Smashing with $\mathbb{\Psi}^p$ or with $\mathbb{T}_p$ before applying $\pi_*$ then yields an action of $\Psi^p$ and of $\tilde{T}_p$ on $E_*(-)$. Consequently, if $E^*$ has Hecke operators, then we get Adams and Hecke operations on $E$-homology as well. 
The Hecke operations in $E$-homology still satisfy relations \eqref{hecke relation 0} and \eqref{hecke relation}. Since cohomological degrees are equal to $-1$ times homological degrees, the Adams operation $\Psi^p$ on $E_{2k}(S^0)$ is multiplication by $p^k$ rather than $p^{-k}$. 
\end{observation}

\begin{remark}\label{adams ops remark}
In Proposition 6 of \cite{MR1037690}, Baker explains how to use the image of the $p$-series of the universal formal group law on $MU_*$ to get that $\Ell^*[P^{-1}]$ has Adams operations in $P$.
Baker's construction also yields, for each $p$, a multiplicative degree-preserving Adams operation $\Psi^p$ on $\elll[\frac{1}{p}]^*$, and its effect on $\elll[\frac{1}{p}]^{2k}(\pt)$ agrees with multiplication by $p^{-k}$ since the same is true on $\Ell[\frac{1}{p}]^{2k}$ and since the natural transformation $\elll[\frac{1}{p}]^* \rightarrow \Ell[\frac{1}{p}]^*$ commutes with the Adams operations and is injective when evaluated on $S^0$.
Consequently, for any set of primes $P$, $\elll[P^{-1}]^*$ has Adams operations in $P$.
\end{remark}

Let now $\mathbf{A}_P$ be the abstract Hecke algebra for $P$, as defined above in Definition \ref{def:abstract Hecke for P}. In Theorem 7 of \cite{MR1037690}, Baker also proved the more difficult result:
\begin{theorem}\label{baker thm on hecke action}
For every set $P$ of primes, $P$-inverted weakly holomorphic elliptic cohomology $\Ell[P^{-1}]^*$ has Hecke operators in $P$. The action of $\tilde{T}_p$ on $\Ell[P^{-1}]^{-2k}(S^0) \cong \Ell[P^{-1}]_{2k}(S^0)$ coincides with the action of the classical Hecke operator $T_p$ on the weight $k$ weakly holomorphic modular forms.
\end{theorem}
Baker's proof works equally well to prove that
\begin{theorem}\label{baker thm on hecke action 2}
For every set $P$ of primes, $P$-inverted holomorphic elliptic cohomology $\elll[P^{-1}]^*$ has Hecke operators in $P$. The action of $\tilde{T}_p$ on $\elll[P^{-1}]^{-2k}(S^0) \cong \elll[P^{-1}]_{2k}(S^0)$ coincides with the action of the classical Hecke operator $T_p$ on the weight $k$ holomorphic modular forms.
\end{theorem}

As a consequence of Theorem \ref{baker thm on hecke action 2} and Observation \ref{brown rep observation}, we have maps\footnote{To be careful: since these maps of spectra are constructed using Brown representability, they are only defined up to homotopy.} of spectra $\Psi^p: \elll[\frac{1}{p}] \rightarrow \elll[\frac{1}{p}]$ and $\tilde{T}_{p^r}: \elll[\frac{1}{p}] \rightarrow \elll[\frac{1}{p}]$ for each prime $p$ and each positive integer $r$, and consequently, given a set $P$ of primes, we have an Adams operator $\Psi^p:\elll[P^{-1}]_* \rightarrow \elll[P^{-1}]_*$ for each $p\in P$ and a Hecke operator $\tilde{T}_n: \elll[P^{-1}]_* \rightarrow \elll[P^{-1}]_*$ for each positive integer $n$ whose prime factors are all in $n$, satisfying relations \eqref{hecke relation 0} and \eqref{hecke relation}. The Adams operator $\Psi^p$ agrees with multiplication by $p^k$ on $\elll[P^{-1}]_{2k}(S^0)$, and so relation \eqref{hecke relation} reduces to the familiar classical relation \begin{align}\label{hecke relation 2} \tilde{T}_{p^{r+2}} &= \tilde{T}_p\tilde{T}_{p^{r+1}} - p^{k-1}\tilde{T}_{p^r}\end{align} on $\elll_{2k}[P^{-1}](S^0)$.

\begin{remark}\label{remark on why 6 is inverted 2}
Baker's Hecke operations are defined on weakly holomorphic elliptic homology $\Ell_*$ and also on holomorphic elliptic homology $\elll_*$, but $6$ is already inverted in the coefficient ring of both of these generalized homology theories (see Definition \ref{def of elll-cohomology}). The primes $2$ and $3$ are not already inverted in the generalized homology theory $TMF_*$ (which agrees with $\Ell_*$ after inverting $6$) and in the generalized homology theories $Tmf_*$ and $tmf_*$ (the latter of which agrees with $\elll_*$ after inverting $6$).

We see some indications that it may be possible to construct stable Hecke operators on $TMF_*$, $Tmf_*$, and $tmf_*$: the idea would be to mimic, in spectral algebraic geometry, Conrad's integral version (from section 4.5 of \cite{MR2311664}) of the well-known geometric construction of Hecke operators. 
Conrad constructed a moduli stack $\mathcal{M}_{\Gamma_0(p)}$ of generalized elliptic curves $E$ over $\mathbb{Z}$ equipped with a cyclic subgroup $G$ of order $p$ on $E^{sm}$ whose Cartier divisor is ample, and finite flat maps $\pi_1,\pi_2: \mathcal{M}_{\Gamma_0(p)} \rightarrow \mathcal{M}_{1,1}$ to the moduli stack of generalized elliptic curves, such that the induced trace map $\Tr: H^0\left( \mathcal{M}_{1,1}; \pi_{1*}\pi_1^*\omega_{\mathcal{M}_{1,1}}^{\otimes k}\right) \rightarrow H^0\left( \mathcal{M}_{1,1}; \omega_{\mathcal{M}_{1,1}}^{\otimes k}\right) \cong M_k$ has the property that its composite
\begin{align} M_k = H^0\left(\mathcal{M}_{1,1}; \omega_{\mathcal{M}_{1,1}}^{\otimes k}\right)
\label{map 1} &\stackrel{\xi^{\otimes k}\circ \pi_2^*}{\longrightarrow}  
  H^0\left(\mathcal{M}_{\Gamma_0(p)}; \omega_{\mathcal{M}_{\Gamma_0(p)}}^{\otimes k}\right) \\
\label{map 2} &\stackrel{\cong}{\longrightarrow}  
  H^0\left( \mathcal{M}_{1,1}; \pi_{1*}\pi_1^*\omega_{\mathcal{M}_{1,1}}^{\otimes k}\right) \\
\label{map 3} &\stackrel{\Tr}{\longrightarrow}
  M_k, \end{align}
where $\xi: \pi_2^*\omega \rightarrow \pi_1^*\omega$ is the pull-back along the universal $p$-isogeny,  agrees with $pT_p$ on the complex fiber. 
It seems quite plausible that Conrad's constructions can also be carried out with the well-known $E_{\infty}$-ring-spectral enrichment of $\mathcal{M}_{1,1}$ (e.g. as in \cite{MR2597740}), especially if one is willing to invert the prime $p$, yielding a topological Hecke operator $\tilde{T}_p: \Tmf[\frac{1}{p}] \rightarrow \Tmf[\frac{1}{p}]$. Our understanding from \cite{daviesinpreparation} is that Jack Davies has obtained positive results in this area.

However, the torsion elements in $\pi_*(TMF)$ and $\pi_*(Tmf)$ and $\pi_*(tmf)$ arise from the $2$-torsion and $3$-torsion in the \'{e}tale cohomology groups $H^n_{\et}(\mathcal{M}_{1,1}; \omega^{\otimes k})$ for $n>0$, arising from $\mathcal{M}_{1,1}$ being a Deligne-Mumford stack with ``stacky'' points with finite isotropy groups of orders divisible by $2$ and $3$. Now the stack $\mathcal{M}_{\Gamma_0(p)}$ is covered by $\mathcal{M}_{\Gamma_1(p)}$, the classifying stack of elliptic curves together with a point of order $p$. For $p>3$ this stack is actually a (smooth) scheme, and so we do not get any of these torsion groups in $H^n_{\et}(\mathcal{M}_{\Gamma_1(p)}; \omega^{\otimes k})$ for $n>0$. The stack $\mathcal{M}_{\Gamma_1(p)}$ is a Galois cover of $\mathcal{M}_{\Gamma_0(p)}$ with Galois group $Aut(\mathbb{F}_p) \cong \mathbb{F}_p^{\times}$, acting on $\mathcal{M}_{\Gamma_1(p)}$ by the ``diamond operators'' which permute the generators in an order $p$ subgroup of a given elliptic curve. In particular, all the torsion in the dimension $>1$ \'{e}tale cohomology groups of $\mathcal{M}_{\Gamma_0(p)}$ comes from the isotropy of the action of the diamond operators, i.e., from the group cohomology of $\mathbb{F}_p^{\times}$. 

Consequently, under the composite of the maps \eqref{map 1}, \eqref{map 2}, and \eqref{map 3}, above, the $2$ and $3$-torsion elements of $H^n_{\et}(\mathcal{M}_{1,1}; \omega^{\otimes k})$, which give rise to $2$- and $3$-torsion elements in $\pi_*(TMF)$ and $\pi_*(Tmf)$, factor through group cohomology groups of $\mathbb{F}_p^{\times}$ acting via diamond operators on $\mathcal{M}_{\Gamma_1(p)}$. So, even if one had topological Hecke operators $\tilde{T}_p$ defined on $TMF$ or $Tmf$ or $tmf$ (i.e., with $2$ and $3$ not inverted), the effect of these topological Hecke operators on $tmf_*$ algebraically comes from the group cohomology of a cyclic group. In addition, the torsion in $\Gamma_0(p)$ is actually trivial for a set of primes $p$ of positive density, and consequently the  action of $\tilde{T}_p$ on 2- and 3-torsion in $tmf_*$ will be trivial for a set of primes of positive density. In particular, from the formulas in \cite[Prop. 1.43]{Shimura} the action of $\tilde{T}_p$ on the $2$-torsion in $tmf_*$ is trivial if $p\equiv 3$ modulo $4$, and the action of $\tilde{T}_p$ on the $3$-torsion in $tmf_*$ is trivial if $p\equiv 2$ modulo $3$. 

For these reasons, the special behavior of the primes $2$ and $3$ in $TMF$ and $Tmf$ seems orthogonal to questions about topological Hecke operators and their eigenforms, so we have not hesitated to simply invert $6$, and work with elliptic homology rather than $tmf$.
\end{remark}

\section{Hecke eigenforms over topological spaces.}
\label{Hecke eigenforms over topological spaces.}

\subsection{Topological Hecke eigenforms}
\label{Topological Hecke eigenforms and...}

Let $X$ be a CW-complex or a bounded-below spectrum, and let $E$ be a spectrum. We have the homological and cohomological Atiyah-Hirzebruch spectral sequences
\begin{align}
\label{homological ahss}   E^2_{s,t} \cong \tilde{H}_s(X; E_t(S^0)) &\Rightarrow E_{s+t}(X)\mbox{\ \ and}\\ 
\nonumber             d^r: E^r_{s,t}                   &\rightarrow E^r_{s-r,t+r-1} \\
\label{cohomological ahss} E_2^{s,t} \cong \tilde{H}^s(X; E^t(S^0)) &\Rightarrow E^{s+t}(X) \\
\nonumber             d_r: E_r^{s,t}                   &\rightarrow E_r^{s+r,t-r+1}.
\end{align}
Spectral sequence \eqref{cohomological ahss} is conditionally convergent, with strong convergence if $X$ is finite-dimensional\footnote{To be absolutely clear about the terminology: when a CW-complex is called ``finite-dimensional,'' this means it is required to have cells in only finitely many dimensions, although it is allowed to have infinitely many cells in individual dimensions. When a CW-complex is ``finite,'' this means it has only finitely many cells {\em tout court}.}.
Meanwhile, spectral sequence \eqref{homological ahss} is always strongly convergent. See section 12 of \cite{MR1718076} for these convergence claims.

It is a theorem of A. Dold (see 14.18 of \cite{MR0198464} or Corollary 2.6 of \cite{MR0254838}, or for a reference more easily found on the Web, \cite{MR1193149}) that the shortest differential in any spectral sequence of the type \eqref{homological ahss} or \eqref{cohomological ahss} is torsion-valued. 
Consequently, if $\tilde{H}^*(X; \mathbb{Z})$ and $E^*(S^0)$ are each torsion-free, then spectral sequence \eqref{cohomological ahss} collapses with no nonzero differentials. 
This leads, without much fuss, to the following result, which is known to experts but which we have not been able to locate in the literature:
\begin{prop}\label{ahss splitting}
Let $X$ be a finite-dimensional CW-complex and let $E$ be a ring spectrum. Suppose that $P$ is a set of prime numbers such that:
\begin{itemize}
\item $\pi_*(E)$ is torsion-free as an abelian group,
\item $\frac{1}{p}\in \pi_0(E)$ for all $p\in P$, and
\item $\tilde{H}_*(X;\mathbb{Z}[P^{-1}])$ is a free graded $\mathbb{Z}[P^{-1}]$-module.
\end{itemize}
Equip the bigraded $\pi_*(E)$-module $\tilde{H}_*(X; \mathbb{Z})\otimes_{\mathbb{Z}} \pi_*(E)$ with the (single) grading by total degree, i.e., $\left(\tilde{H}_*(X; \mathbb{Z})\otimes_{\mathbb{Z}} \pi_*(E)\right)_n = \bigoplus_{s+t=n}\tilde{H}_s(X; \mathbb{Z})\otimes_{\mathbb{Z}} \pi_t(E)$. Then we have an isomorphism of graded $E_*$-modules
\[ E_*(X) \cong \tilde{H}_*(X; \mathbb{Z})\otimes_{\mathbb{Z}} \pi_*(E).\]

Similarly, 
we have an isomorphism of graded $E^*(S^0)$-modules
\[ E^*(X) \cong \hom_{\mathbb{Z}}\left(\tilde{H}_*(X; \mathbb{Z}),E^*(S^0)\right),\]
i.e., $E^n(X) \cong \bigoplus_s \hom_{\mathbb{Z}}\left( \tilde{H}_s(X;\mathbb{Z}),E^{n-s}(S^0)\right).$
\end{prop}
\begin{proof}
Dold's theorem, described above, implies that spectral sequence \eqref{homological ahss} collapses with no nonzero differentials.
There still remains the extension problem: the $E^{\infty}$-page of spectral sequence \eqref{homological ahss} is isomorphic, as a bigraded $\pi_*(E)$-module, to $\tilde{H}_*(X;\mathbb{Z})\otimes_{\mathbb{Z}}\pi_*(E)$, but this $E^{\infty}$-page is only the associated graded of the filtration on $E_*(X)$ induced by the skeletal filtration of $X$. 

However, the extension problem is not difficult to solve. Since $X$ is finite-dimensional, there exists some $n$ such that the $n$-skeleton of $X$, $X^n$, agrees with $X$. So the $E^{\infty}$-page of spectral sequence \eqref{homological ahss} is concentrated to the left of the $s=n+1$-line (``left'' here refers to the convention of drawing spectral sequences \eqref{homological ahss} and \eqref{cohomological ahss} with the Serre convention, so that $s$ is the horizontal axis and $t$ is the vertical axis), i.e., all of $E_*(X)$ is concentrated in Atiyah-Hirzebruch filtration $n$, the image of the homomorphism $E_*(X^n) \rightarrow E_*(X)$. So the projection map from $E_*(X^n)$ to the $s=n$ line in the $E^{\infty}$-term of spectral sequence \eqref{homological ahss} is a projection from $E_*(X)$ to that line. Since $\pi_*(E)$ is torsion-free and every prime in $P$ is inverted in $\pi_*(E)$, we have that $\pi_*(E)$ is flat as a $\mathbb{Z}[P^{-1}]$-module, so by the universal coefficient sequence for homology, we have that $\tilde{H}_*(X; \pi_*(E)) \cong \tilde{H}_*(X; \mathbb{Z}[P^{-1}]) \otimes_{\mathbb{Z}[P^{-1}]} \pi_*(E)$, so freeness of $\tilde{H}_n(X; \mathbb{Z}[P^{-1}])$  as an $\mathbb{Z}[P^{-1}]$-module implies freeness of $\tilde{H}_n(X; \pi_*(E))$ as an $\pi_*(E)$-module. Collapse of the spectral sequence implies that the $s=n$ line in the $E^{\infty}$-term is isomorphic to the free $E_*$-module $\tilde{H}_n(X; \pi_*(E))$, so the projection of $\tilde{H}_*(X; \mathbb{Z}[P^{-1}])$ to the $s=n$-line in the $E^{\infty}$-term splits. 

That handles the first (counting from the right-hand side in the $E^{\infty}$-term of spectral sequence \eqref{homological ahss}) extension problem: the extension is trivial. The rest are handled by an easy induction, by essentially the same argument as in the previous paragraph: if we already have shown that the first $j$ extension problems are trivial, then the canonical inclusion of $\pi_*(E)$-modules $\im \left(E_*(X^{n-j})\rightarrow E_*(X)\right) \hookrightarrow E_*(X)$ is split.
The inclusion 
\begin{equation}\label{incl 230}\im \left(E_*(X^{n-j-1})\rightarrow E_*(X)\right) \hookrightarrow \im \left(E_*(X^{n-j})\rightarrow E_*(X)\right)\end{equation} has cokernel the $s=n-j$-line in the $E^{\infty}$-page of spectral sequence \eqref{homological ahss}, which is a free $\pi_*(E)$-module by the same argument as in the previous paragraph. So the inclusion of $\pi_*(E)$-modules \eqref{incl 230} splits, finishing the inductive step.

The cohomological claim is similar: in that case, we do not need the finite-dimensionality of the CW-complex $X$ for the extension problem (since we begin solving extension problems on the left-hand vertical line in $E_{\infty}$-term of spectral sequence \eqref{cohomological ahss}, rather than on the right-hand vertical line as in spectral sequence \eqref{homological ahss}. Instead, in the cohomological case, the finite-dimensionality of $X$ is simply used to get strong convergence of spectral sequence \eqref{cohomological ahss}. Collapse of the spectral sequence and triviality of extension problems in its $E_{\infty}$-term is otherwise the same as in the homological case. The isomorphism between $\hom_{\mathbb{Z}[P^{-1}]}\left(\tilde{H}_*(X; \mathbb{Z}[P^{-1}]),E^*(S^0)\right)$ and $H^*(X; E^*(S^0))$ is, of course, just the universal coefficient sequence for cohomology, since $\Ext^1_{\mathbb{Z}[P^{-1}]}\left(\tilde{H}_*(X; \mathbb{Z}[P^{-1}]),E^*(S^0)\right)$ vanishes.
\end{proof}

\begin{observation}\label{hecke action observation}
Let $P$ be a set of primes containing $2$ and $3$, and let $X$ be a finite-dimensional CW-complex such that $\tilde{H}_*(X; \mathbb{Z}[P^{-1}])$ is a free graded $\mathbb{Z}[P^{-1}]$-module. 
By Proposition \ref{ahss splitting}, the holomorphic elliptic homology $\elll[P^{-1}]_*(X)$ of $X$ is isomorphic to $\tilde{H}_*(X; \mathbb{Z})\otimes_{\mathbb{Z}} \elll[P^{-1}]_*(S^0)$ as a graded $\elll[P^{-1}]_*(S^0)$-module, and by naturality (in the choice of space/spectrum) of the Hecke action, the action of Hecke operators on the elliptic homology of the $n$-skeleton $X^n$ of $X$ lands in the elliptic homology of $X^n$.
Furthermore, from the proof of Proposition \ref{ahss splitting}, we obtained the isomorphism $\elll[P^{-1}]_k(X) \cong \oplus_{i+j=k}\tilde{H}_i(X; \mathbb{Z})\otimes_{\mathbb{Z}} \elll[P^{-1}]_j(S^0)$ by identifying $\tilde{H}_i(X; \mathbb{Z})\otimes_{\mathbb{Z}}\elll[P^{-1}]_j(S^0)$ with $\elll[P^{-1}]_{i+j}$ of a wedge of $i$-spheres, i.e., $\elll[P^{-1}]_{j}$ of a wedge of $0$-spheres, i.e., a direct sum of copies of $\elll[P^{-1}]_{j}(S^0)$, on which the Hecke diagonal action is a block sum of copies of the classical Hecke action on weight $j$ modular forms. Consequently the action of each Hecke operator $T\in \mathbf{A}_P$ on \[ \elll[P^{-1}]_k(X) \cong \oplus_{i+j=k}\tilde{H}_i(X; \mathbb{Z})\otimes_{\mathbb{Z}} \elll[P^{-1}]_j(S^0)\] is {\em upper triangular}: $\tilde{T}$ sends an element of $\tilde{H}_{k-i}(X; \mathbb{Z})\otimes_{\mathbb{Z}} \elll[P^{-1}]_i(S^0) \subseteq \elll[P^{-1}]_k(X)$ to the block sum of $\dim_{\mathbb{Z}[P^{-1}]}\tilde{H}_{k-i}(X;\mathbb{Z}[P^{-1}])$ copies of the classical action of the Hecke operator $T$ on  weight $i$ modular forms, {\em plus terms in $\elll[P^{-1}]_k(X)$ of lower skeletal/Atiyah-Hirzebruch filtration,} i.e., terms in $\tilde{H}_{k-j}(X; \mathbb{Z})\otimes_{\mathbb{Z}} \elll[P^{-1}]_j$ with $j>i$. The conclusion here is that, after extending scalars to $\Qbar$, the eigenspaces of the action of the abstract Hecke algebra $\mathbf{A}_P$ on $\elll[P^{-1}]_*(X)$ are contained in the eigenspaces of the diagonal classical action of $\mathbf{A}_P$ on $\tilde{H}_*(X; \mathbb{Z})$ tensored with the ring $\elll[P^{-1}]_*(S^0)\otimes \Qbar$ of modular forms over $\Qbar$. 
\end{observation}

This leads us to Definition \ref{def of eigenform}.
Let $P\subseteq \Pi$ be a set of primes containing $2$ and $3$, and let $X$ be a finite-dimensional CW-complex such that $\tilde{H}_*(X; \mathbb{Z}[P^{-1}])$ is a free graded $\mathbb{Z}[P^{-1}]$-module. We continue to write $M_*$ for the graded ring of level 1 holomorphic modular forms with coefficients in $\mathbb{Z}[P^{-1}]$. 
Let $k \in \frac{1}{2}\ZZ $ be a half-integer and let $\mathbf{A}_{P}$ be the abstract Hecke algebra for $P$ acting on $X$. 

\begin{definition}
\label{def of eigenform} 
By a {\em topological Hecke eigenform over $X$ of weight $k \in \frac{1}{2}\ZZ $ 
for the primes in $P$} we mean 
an element $f\in \elll_{2k}(X)\otimes_{\mathbb{Z}}\mathcal{O}_{\Qbar}[P^{-1}]$ which is a common  eigenvector for the action of $\mathbf{A}_{P}$. That is, 
$$
\tilde{T}_nf = \lambda(\tilde{T}_n)f
$$
for all $n$ whose prime factors are each in $P$, and for some eigencharacter $\lambda: \mathbf{A}_{P} \rightarrow \mathcal{O}_{\Qbar}[P^{-1}]
$ with values in the ring $\mathcal{O}_{\Qbar}$ of algebraic integers. 
\end{definition}

Under the isomorphism
\begin{equation}
\label{splitting iso 34094}
\coprod_{i} \tilde{H}_{2k-i}(X; \mathbb{Z})\otimes_{\mathbb{Z}} M_i \otimes_{\mathbb{Z}} \mathcal{O}_{\Qbar}\cong \elll_{2k}(X)\otimes_{\mathbb{Z}} \mathcal{O}_{\Qbar}[P^{-1}],\end{equation}
every topological Hecke eigenform over $X$ corresponds to a vector
of classical holomorphic modular forms, of various weights depending on the homology groups of $X$. Of course, not every vector of classical modular forms of those weights will correspond to a topological Hecke eigenform over $X$. 
One of our tasks in this section and the next is to identify the conditions on such a vector of classical modular forms which cause it to be a topological Hecke eigenform over $X$. These conditions depend on the choice of topological space $X$. We begin with the simplest case, in the following example.
 
\begin{example}\label{example examples of hecke eigenforms}
The trivial cases are spheres: the weight $k$ Hecke eigenforms over $S^0$ for $P = \Pi$ (or $P$ cofinite) are exactly the classical weight $k$ Hecke eigenforms, since 
$$
\elll_{2k}(S^0)\otimes_{\mathbb{Z}} \mathcal{O}_{\Qbar}[P^{-1}] \simeq M_k \otimes_{\mathbb{Z}} \mathcal{O}_{\Qbar}
$$
and the action of the topological Hecke operators $\tilde{T}_n\in \mathbf{A}_P$ on $\elll_{2k}(S^0)\otimes_{\mathbb{Z}} \mathcal{O}_{\Qbar}[P^{-1}]$
coincides with that of the classical Hecke operators on $M_k$. Note that the degree in elliptic homology is twice the weight. This is the reason why we need to allow $k\in \frac{1}{2}\ZZ$ in the definition of a topological eigenform.

Only slightly more generally, the weight $k \in \frac{1}{2}\ZZ$ topological Hecke eigenforms over $S^i$ for all primes are exactly the classical weight $k-\frac{i}{2}$ Hecke eigenforms  (See Remark \ref{remark on various topics} for some comments on the appearance of a half-integer weight here.) This follows simply from the fact that Baker's topological Hecke operators are {\em stable}. See Remark \ref{remark on why 6 is inverted 1} for some discussion of stability of the topological Hecke operators.
\end{example}

\begin{remark}\label{remark on various topics}
Here are some remarks about Definition \ref{def of eigenform} which we expect may be useful to some readers:
\begin{enumerate}
\item
In Definition \ref{def of eigenform}, the only reason for the base change to $\mathcal{O}_{\overline{\mathbb{Q}}}$ is to ensure that all eigenvalues of the action of the Hecke operators on $\elll_{2k}(X)$ are present in the ground ring. If $X$ is assumed to be a finite-type CW complex, then $\elll_{2k}(X)$ is finite-dimensional for each $k$, so it is not necessary to base change all the way to $\mathcal{O}_{\overline{\mathbb{Q}}}$: all topological Hecke eigenforms over $X$ of weight $k$ will already be present after base change to $\mathcal{O}_K$ for some finite extension $K$ of $\mathbb{Q}$. So the presence of $\mathcal{O}_{\overline{\mathbb{Q}}}$ in Definition \ref{def of eigenform}, rather than merely the ring of integers in some sufficiently large finite extension of $\mathbb{Q}$, is not essential.
\item
When $X$ has nontrivial homology in odd degrees, Definition \ref{def of eigenform} allows for the possibility of half-integer-weight topological Hecke eigenforms. For example, in the case $X = S^1$, every classical Hecke eigenform of weight $k$ also describes a topological Hecke eigenform of weight $k+\frac{1}{2}$ over $S^1$. Such half-integer weight topological Hecke eigenforms do not directly relate to holomorphic functions on the upper half-plane which satisfy a modularity law of half-integer weight, nor should be mistaken for the ``classical half-integer weight modular forms'' (\cite{MR332663}). A suitable generalized homology theory of ``half-integer weight topological modular forms,'' which when evaluated on spheres includes modular forms of half-integer weight as studied in \cite{MR332663}, is perhaps constructible after inverting $2$ (and perhaps, with much more difficulty, also without inverting $2$) using methods of spectral algebraic geometry applied to the metaplectic stacks and theta multiplier line bundles considered in \cite{candelorithesis}, but this has not yet been written down anywhere or studied systematically, and is orthogonal to the ideas presented in the present paper.
\end{enumerate}
\end{remark}

\subsection{Computation of Hecke eigenforms over wedges of spheres.}
\label{Computation of Hecke...}

Since the cohomology of any finite CW-complex $X$ with torsion-free homology coincides with the cohomology of a wedge of spheres, one sphere for each element in a $\mathbb{Z}$-linear basis for $\tilde{H}^*(X; \mathbb{Z})$, we see a relationship between the homotopical properties of $X$ and the supply of topological Hecke eigenforms over $X$. That is, suppose we fix a connective graded free $\mathbb{Z}[P^{-1}]$-algebra $B$, and we study all the finite CW-complexes $X$ such that $\tilde{H}^*(X; \mathbb{Z}[P^{-1}]) \cong B$. The homotopically simplest such CW-complex $X$ is simply a wedge of spheres: it is ``homotopically simple'' in the sense that all the attaching maps in the CW-complex are nulhomotopic. If $X$ is not a wedge of spheres, then some of the attaching maps in the CW-decomposition of $X$ must not be nulhomotopic, and consequently the Hecke action on the cofibers of those attaching maps may have nonzero off-diagonal terms. Consequently, not every vector (indexed by the homology of $X$) of classical Hecke eigenforms will be a topological Hecke eigenform on $X$, since an upper-triangular matrix can have smaller eigenspaces than a diagonal matrix. As a slogan, ``the more non-nulhomotopic attaching maps there are in a CW-decomposition for $X$, the fewer Hecke eigenforms we expect $X$ to have.''


For a finite wedge of spheres, however, we have a complete and simple understanding of the topological Hecke eigenforms. To state the result, given a topological eigencharacter $\lambda$ on a space $X$, denote by $V_X(\lambda)$ the corresponding eigenspace, and similarly denote by $V(\nu(\lambda))$  the corresponding eigenspace of classical eigenforms corresponding to its natural character.

\begin{prop}
\label{prop:wedges}
\label{examples of hecke eigenforms}
Let $P$ be a cofinite set of primes, and suppose that $X$ is a finite wedge of spheres\footnote{This same result remains true under the weaker assumption that $X$ splits as a finite wedge of spheres after inverting the primes in $P$.}. Let $\lambda: \mathbf{A}_{P}\rightarrow \overline{\mathbb{Q}}$ be an eigencharacter.
Write $V(\lambda)_*$ for the graded $\mathbb{Z}[P^{-1}]$-submodule of $M_*\otimes_{\mathbb{Z}} \mathbb{Z}[P^{-1}]$ consisting of eigenforms\footnote{We remind the reader that here, and everywhere else throughout this paper, all modular forms are assumed to be of level $1$. Analogous results at higher level are provable by similar methods.} for the action of $T_p$, for all $p\in P$, with eigencharacter $\lambda$. 
Write $V_X(\lambda)_*$ for the graded $\mathbb{Z}[P^{-1}]$-submodule of $\elll_*(X)\otimes_{\mathbb{Z}} \mathbb{Z}[P^{-1}]$ consisting of eigenforms for the action of $\tilde{T}_p$, for all $p\in P$, with eigencharacter $\lambda$.
Then we have isomorphisms of graded $\mathbb{Z}[P^{-1}]$-modules
\begin{align}
\label{iso 301} HH^0\left(\mathbf{A}_P; \left(\elll_*(X)\otimes_{\mathbb{Z}} \mathbb{Z}[P^{-1}]\right)^{\lambda} \right)
  &\cong V_X(\lambda)_* \\
\label{iso 302}  &\cong D(V(\lambda)_{*})\otimes_{\mathbb{Z}} \tilde{H}_{*}(X;\mathbb{Z})\otimes_{\mathbb{Z}} \mathbb{Z}[P^{-1}] \\
\label{iso 303}  &\cong HH^0\left(\mathbf{A}_P; \left(M_*\otimes_{\mathbb{Z}} \mathbb{Z}[P^{-1}]\right)^{\lambda} \right),
\end{align}
where $D(V(\lambda)_*)$ is the graded $\mathbb{Z}[P^{-1}]$-module $V(\lambda)_*$ with all grading degrees multiplied by $2$.
\end{prop}


\begin{proof}
Isomorphisms \eqref{iso 301} and \eqref{iso 303} were explained already in \cref{Hochschild cohomology and derived eigenforms} preceding Definition \ref{def:derivedEigenform}.
Isomorphism \eqref{iso 302} is simply due to $X$ being a wedge of spheres, so that the Hecke operators act diagonally on the summands in the decomposition 
\begin{align*}
 \elll_{2k}(X)\otimes_{\mathbb{Z}}\overline{\mathbb{Q}} 
  &\cong \coprod_i \elll_i \otimes_{\mathbb{Z}} \tilde{H}_{2k-i}(X; \mathbb{Q})\otimes_{\mathbb{Q}} \overline{\mathbb{Q}} \\
  &\cong \coprod_i M_{i/2} \otimes_{\mathbb{Z}} \tilde{H}_{2k-i}(X; \mathbb{Q})\otimes_{\mathbb{Q}} \overline{\mathbb{Q}} 
\end{align*}
obtained from Proposition \ref{ahss splitting}.
%
\end{proof}


\subsection{Multiplicity-one spaces}

More interesting examples of topological eigenforms originate from spaces that satisfy a suitable ``multiplicity one'' property, defined as follows.

\begin{definition}\label{def of multiplicity one}
Given a set $P$ of primes, we will say that a topological space $X$ {\em has multiplicity one (for $P$)} if, for each $k\in \frac{1}{2}\ZZ$ and each eigencharacter $\lambda: \mathbf{A}_P\rightarrow \mathcal{O}_{\Qbar}[P^{-1}]$ of the abstract Hecke algebra, the rank of the $\mathcal{O}_{\Qbar}[P^{-1}]$-module of eigenvectors in $ell_{2k}(X)\otimes \mathcal{O}_{\Qbar}[P^{-1}]$ of eigencharacter $\lambda$ is at most one.
\end{definition}

The classical multiplicity one theorem for cuspidal newforms is fundamental in the theory of modular forms. Even still, multiplicity one results have an additional importance in our topological setting which does not appear to have any analogue in the classical number-theoretic setting: if an element $f$ of $\elll_{2k}(X)$ is an eigenvector for $\tilde{T}_{\ell}$ for each prime $\ell$ in some set $P$ of primes, it is not immediately clear that $f$ is a topological Hecke eigenform for $P$, because $f$ might not be an eigenform for (for example) $\tilde{T}_{\ell^2} = \tilde{T}_{\ell}\tilde{T}_{\ell} - \frac{1}{p}\Psi^p$, since $f$ might not be an eigenform for the action of the Adams operation $\Psi^p$. This trouble is avoided whenever we have a multiplicity one theorem, as we show below in Prop. \ref{multiplicity one is convenient}.

\begin{prop}\label{multiplicity one is convenient} 
Let $P$ be a set of prime numbers and let $X$ be a topological space which has multiplicity one for $P$. Then the topological Hecke eigenforms of weight $k$ for $P$ are precisely the elements $f\in \elll_{2k}(X)\otimes_{\mathbb{Z}}\mathcal{O}_{\overline{\mathbb{Q}}}[P^{-1}]$ such that $f$ is an eigenform of $\tilde{T}_{\ell}$ for each prime $\ell\in P$.
\end{prop}
\begin{proof}
Suppose $f$ is a Hecke eigenform for the action of $\tilde{T}_{p}$ for all $p\in P$, and suppose that $\tilde{T}_{\ell}f = \lambda_{\ell}f$ for some particular prime $\ell\in P$. Then we have $\tilde{T}_{\ell}\Psi^{\ell}f = \Psi^{\ell}\tilde{T}_{\ell}f = \Psi^{\ell}\lambda_{\ell}f = \lambda_{\ell}\Psi^{\ell}f$, i.e., $\Psi^{\ell}f$ is an eigenform for $\tilde{T}_{\ell}$  with eigenvalue $\lambda_{\ell}$. By multiplicity one, $\Psi^{\ell}f$ is a scalar multiple of $f$, i.e., $f$ is an eigenform for the action of the Adams operation $\Psi^{\ell}$. This argument applies to every prime $\ell\in P$, so $f$ is also an eigenform for the action of every Adams operator $\Psi^{\ell}$ with $\ell\in P$. Consequently, by the equation $\tilde{T}_{\ell^{r+2}} = \tilde{T}_{\ell}\tilde{T}_{\ell^{r+1}} - \frac{1}{\ell}\Psi^{\ell}\tilde{T}_{\ell^r}$ for each $r\geq 0$, $f$ is also an eigenform for the action of $\tilde{T}_{\ell^r}$ for every prime power $\ell^r$ with $\ell\in P$. Finally, since $\tilde{T}_{mn} = \tilde{T}_m \tilde{T}_n$ for coprime $m,n$, $f$ is an eigenform for the action of every Hecke operator $\tilde{T}_n$ with all prime factors of $n$ contained in $P$.
\end{proof}

\begin{example}
\leavevmode
\begin{itemize}
\item
If $P$ is cofinite, then the zero-sphere $S^0$ has multiplicity one for $P$. Indeed, a topological Hecke eigenform of weight $k$ on $S^0$ for $P$ is just a classical weight $k$ Hecke eigenform $f$ for all Hecke operators $T_p$, $p\in P$, so in this case the multiplicity one property is classical. Conjecturally, it even suffices to take $P\neq \emptyset$: Maeda's conjecture for level one forms would imply that the characteristic polynomial of $T_p$ has no repeated roots, for any prime $p\in \Pi$. 
\item
As a consequence,
by the stability of elliptic homology and the fact that Hecke operators are stable operators, all spheres have multiplicity one for all cofinite $P$. This is another example of a deduction that relies on the stability of the Hecke operators, which is a feature of Baker's Hecke operators; we point this out because some examples of {\em unstable} topological Hecke operators have also been studied, e.g. in \cite{MR1637129} and in \cite{MR2219307}.
\item In Theorem \ref{multiplicity one thm} we generalize the sphere examples greatly: if $X$ is a finite CW complex with torsion-free homology and at most one cell in each dimension, and $P$ is cofinite, then $X$ has multiplicity one for $P$.
\item On the other hand, a finite CW complex with more than one cell in some dimension is unlikely to have multiplicity one, even if the set of primes $P$ is assumed to be all primes. The simple example is when $X$ is the wedge product of two copies of $S^0$: then the eigenspace of each eigencharacter $\lambda$ on $\elll_{2k}(X)\otimes_{\mathbb{Z}}\mathcal{O}_{\mathbb{Q}}[P^{-1}]$ is isomorphic to a direct sum of two copies of the classical eigenspace of $\lambda$ on $M_k\otimes_{\mathbb{Z}}\mathcal{O}_{\mathbb{Q}}[P^{-1}]$, i.e., is either trivial or two-dimensional, contradicting the multiplicity one property.
\end{itemize}
\end{example}

Proposition \ref{prop:HH0}, and the density argument in its proof, plays a central role in much of the rest of this paper. One consequence is a ``multiplicity one'' theorem for a large class of CW-complexes:
\begin{theorem}\label{multiplicity one thm}
If $X$ is a finite CW-complex with torsion-free homology and at most one cell in each dimension, and $P$ is cofinite, then $X$ has multiplicity one for $P$. 
\end{theorem}
\begin{proof}
Let $X^0\subseteq X^1 \subseteq \dots \subseteq X^n$ be the skeleta in a minimal CW-decomposition (i.e., one which does not include any cells which are capped off) of $X$. 
Then the map $\elll_*(X^j) \rightarrow \elll_*(X)$ is injective for each $j$, by Proposition \ref{ahss splitting}, and this map is equivariant with respect to the topological Hecke operators. Consequently, in the skeletal/Atiyah-Hirzebruch filtration on $H_*(X; \mathbb{Z})\otimes_{\mathbb{Z}} \elll_*(S^0)\cong \elll_*(X)$, the action of the Hecke operators sends elements in filtration $j$ to elements in filtration $j$. That is, the action of each Hecke operator is given by an upper-triangular matrix, after we choose a generator for each $H_j(X; \mathbb{Z})$ and consequently a basis in which to write the Hecke operators as matrices.

Now we use some elementary linear algebra: if $M$ is an upper-triangular square matrix, and $M^{\prime}$ is the same matrix but with all entries not on the main diagonal set to zero, then the eigenvalues of $M$ are the same as those of $M^{\prime}$, and furthermore, for each eigenvalue $\lambda$ of $M$, the dimension of the $\lambda$-eigenspace of $M$ is at most the dimension of the $\lambda$-eigenspace of $M^{\prime}$. 

If we let $M$ be the upper triangular matrix expressing the action of a Hecke operator $T_{\ell}$, then $M^{\prime}$ is (with respect to the same basis for $H_*(X;\mathbb{Z})$) the matrix expressing the action of the Hecke operator $T_{\ell}$ on a CW-complex with the same cells as $X$ but in which all attaching maps are nulhomotopic. 
So, if we have multiplicity one for a finite wedge of spheres of pairwise distinct dimensions, then we have multiplicity one for the space $X$ as described in the statement of the theorem.

Consequently, all we have left to do is to prove multiplicity one for finite wedges of spheres of pairwise distinct dimensions: but this is a simple corollary of Prop. \ref{prop:wedges}.
\end{proof}

\section{Topological Hecke eigenforms over CW-complexes with two cells.}

\label{Extending a Hecke eigenform...}
\subsection{Extending a Hecke eigenform from a cell.}
Our next task is to carry out computations, in the simplest nontrivial cases, to see exactly what the relationship is between the properties (eigenvalues, dimension counts, etc.) of topological Hecke eigenforms over $X$ and the homotopical properties of the space $X$.

The simplest case of topological Hecke eigenforms is the case of spheres and wedges of spheres. That case was already handled completely in Proposition \ref{prop:wedges}. The next case is that of a $2$-cell complex with non-nulhomotopic attaching map, i.e., a suspension of the homotopy cofiber $\cof f$ of a stable\footnote{Since elliptic homology is a generalized homology theory and Baker's Hecke operations are stable operations, we have a natural isomorphism $\elll_*(\Sigma X) \cong \elll_{*-1}(X)$, i.e., $\elll^*$ turns (de)suspensions into shifts of grading, and consequently the collection of Hecke eigenforms over $X$ is a {\em stable} homotopy invariant of $X$. In particular, only the {\em stable} homotopy class of an attaching map in a CW-complexes has an effect on the Hecke eigenforms over that CW-complex.} map $f: S^d \rightarrow S^0$. We refer the reader to the introduction to this paper for an exposition on the relationship between $2$-cell complexes and the stable homotopy groups of spheres.

To deduce a relationship between $f\in \pi_d(S^0)$ and number-theoretic properties of the topological Hecke eigenforms over $\cof f$, we need to be given $f$ in some kind of understandable fashion. One way for an element of $\pi_d(S^0)$ to be given to us is as an element of the $E_2$-term of the Adams-Novikov spectral sequence
\begin{align}
 \label{anss E2 1}  E_2^{s,t}&\cong \Ext^{s,t}_{\gr\ MU_*MU-comod}(MU_*,MU_*) \\
 \label{anss E2 2}           &\cong \Cotor^{s,t}_{\gr\ MU_*MU-comod}(MU_*,MU_*)\\
 \label{anss E2 3}           &\cong H^{s}_{fl}(\mathcal{M}_{fg};\omega^{\otimes t/2})\\
 \label{anss E2 4}                   &\Rightarrow \pi_{t-s}(S^0),\end{align} 
with $t-s = d$. 
The bigraded abelian group \eqref{anss E2 1} is relative $\Ext$ in the abelian category of graded comodules\footnote{Throughout, we will write ``comodule'' as shorthand for ``left comodule.'' Recall that, given a commutative Hopf algebroid $(A,\Gamma)$, a {\em left $\Gamma$-comodule} is a left $A$-module $M$ equipped with a left $A$-linear map $M \rightarrow \Gamma\otimes_AM$ which is counital and coassociative. If $\Gamma$ is flat over $A$, then the category of left $\Gamma$-comodules is abelian and has enough relative injectives. Appendix 1 of \cite{MR860042} is the standard reference for these definitions and results.} over the Hopf algebroid $MU_*MU$. These $\Ext$ groups are taken relative to the class of $MU_*MU$-comodules tensored up from $MU_*$. A good reference for relative $\Ext$ is chapter IX of \cite{MR1344215}. Meanwhile, \eqref{anss E2 2} is the derived functors of the cotensor product in the category of graded comodules over the Hopf algebroid $MU_*MU$, and \eqref{anss E2 3} is the flat cohomology of the moduli stack of one-dimensional formal groups over $\Spec \mathbb{Z}$; see \cite{MR860042} for \eqref{anss E2 1} and \eqref{anss E2 2}, and \cite{coctalos} for \eqref{anss E2 3}.

The stable homotopy groups of spheres are finitely generated abelian groups, and in positive degrees they are finite; hence $\pi_n(S^0) \cong \bigoplus_p \left( \pi_n(S^0)\right)_{(p)}$ for all $n>0$, and so $p$-local methods are usually used to calculate stable homotopy groups of spheres. In particular, for each prime $p$, the $p$-localization $MU_{(p)}$ of the complex bordism spectrum $MU$ splits as a wedge of suspensions of a smaller ring spectrum $BP$ (which depends on the prime $p$, but the choice of $p$ is suppressed from the notation for $BP$), and the $p$-localization of spectral sequence \eqref{anss E2 4} is isomorphic to the $BP$-Adams spectral sequence
\begin{align}
 \label{local anss E2 1}  E_2^{s,t}&\cong \Ext^{s,t}_{\gr\ BP_*BP-comod}(BP_*,BP_*) \\
 \label{local anss E2 2}           &\cong \Cotor^{s,t}_{\gr\ BP_*BP-comod}(BP_*,BP_*)\\
 \label{local anss E2 3}           &\cong H^{s}_{fl}(\mathcal{M}_{fg}\times_{\Spec\mathbb{Z}}\Spec\mathbb{Z}_{(p)};\omega^{\otimes t/2})\\
 \label{local anss E2 4}                   &\Rightarrow \pi_{t-s}(S^0)_{(p)},\end{align} 
where $BP$ denotes $p$-local Brown-Peterson homology, and \eqref{local anss E2 3} is the flat cohomology of the moduli stack of one-dimensional formal groups over $\Spec \mathbb{Z}_{(p)}$. Spectral sequences \eqref{anss E2 4} and \eqref{local anss E2 4} are both called ``the Adams-Novikov spectral sequence,'' although the second is a localization of the first. In this article  we will refer to \eqref{local anss E2 4} as ``the Adams-Novikov spectral sequence'' or ``the ANSS,'' using the phrase ``the global ANSS'' for spectral sequence \eqref{anss E2 4}.

Now suppose we have an element $f\in \pi_{d}(S^0)_{(p)}$. Then $f$ comes from some element\footnote{Possibly more than one, since the ANSS at $p=2$ has many nonzero differentials, and at odd primes the ANSS has nonzero differentials starting with the $E_{2p-1}$-term.} in $\Ext^{s,s+d}_{\gr\ BP_*BP-comod}(BP_*,BP_*)$ for some $s$. One can ask what the least $s$ is such that $f$ comes from an element in $\Ext^{s,s+d}_{\gr\ BP_*BP-comod}(BP_*,BP_*)$. This numerical invariant of elements in $\pi_{d}(S^0)_{(p)}$ is called the {\em $BP$-Adams degree}, and it has an excellent topological description: an element $f\in \pi_{d}(S^0)_{(p)}$ is in $BP$-Adams degree $\geq m$ if and only if $f: S^d \rightarrow S^0$ is the composite of at least $m$ maps of spectra, each of which becomes nulhomotopic after smashing with $BP$. So, for example, $f$ is in $BP$-Adams degree $0$ if and only if $f$ induces a nonzero map in $BP$-homology $\Sigma^dBP_* \cong BP_*(S^d) \rightarrow BP_*(S^0) \cong BP_*$. We refer to the elements in spectral sequence \eqref{local anss E2 4} for a fixed value of $s$ as {\em the Adams-Novikov $s$-line.}

The first calculation one makes with the Adams-Novikov spectral sequence (see chapter 4 of \cite{MR860042}) is that $\Ext^{0,t}_{\gr\ BP_*BP-comod}(BP_*,BP_*)$ is isomorphic to $\mathbb{Z}_{(p)}$, concentrated in $t=0$. Consequently the Adams-Novikov $0$-line is almost trivial: it consists of a single copy of $\mathbb{Z}_{(p)}$, concentrated in bidegree $(0,0)$. (As one might expect, at every prime $p$, this copy of $\mathbb{Z}_{(p)}$ survives the spectral sequence to become $\pi_0(S^0)_{(p)} \cong \mathbb{Z}_{(p)}$.) So, if we take the homotopy cofiber of an element $f\in \pi_d(S^0)_{(p)}$ of $BP$-Adams degree $0$, then the resulting spectrum has holomorphic elliptic homology simply the mod $n$ reduction of the ring of holomorphic modular forms, for some $n\in \mathbb{Z}_{(p)}$, with the topological Hecke operators acting simply as the mod $n$ classical Hecke operators.

So we move on to elements $f\in \pi_d(S^0)_{(p)}$ of $BP$-Adams degree $1$. The ANSS $1$-line is completely calculated; we give the answer here for all $p>2$, where it has a particularly clean form. See \cite{MR737778} for this calculation. We have  
\begin{align}
\label{novikov 1-line} \Ext^{1,t}_{\gr\ BP_*BP-comod}(BP_*,BP_*) &\cong \left\{ \begin{array}{ll}
   \mathbb{Z}/p^{1+\nu_p(t)}\mathbb{Z} & \mbox{\ if\ } 2(p-1)\mid t>0 \\
   0 &\mbox{\ otherwise,}
\end{array}\right.
\end{align}
and when $p>2$ the Adams-Novikov $1$-line neither supports nor is hit by nonzero differentials. Consequently \eqref{novikov 1-line} is a complete picture of the Adams-Novikov $1$-line, at all pages of the ANSS, at odd primes. 

A generator for 
$E_2^{1,2n(p-1)}$ is given by the $1$-cocycle in the $(BP_*,BP_*BP)$ cobar complex
\begin{equation}\label{cocycle rep 10943} \sum_{i=1}^n \binom{n}{i}p^{i-1-\nu_p(n)}  v_1^{n-i} t_1^{i}. \end{equation}

The above cocycle representative is not difficult to calculate (see \cref{appendix on cobar complexes} for the relevant tools), and is certainly not original, and although it probably occurs somewhere in the existing literature, we do not know specifically where.

At odd primes, the cohomology class of \eqref{cocycle rep 10943} in $H^1\left(\mathcal{M}_{fg}\times_{\Spec\mathbb{Z}}\Spec\mathbb{Z}_{(p)};\omega^{n(p-1)}\right)$ survives the Adams-Novikov spectral sequence to generate a summand $\mathbb{Z}/p^{1+\nu_p(n)}\mathbb{Z}$ of $\pi_{(2p-2)n +1}(S^0)$. We refer to the cohomology class of \eqref{cocycle rep 10943} as $v_1^n\alpha_1$; one can sort through the relationship between the cocycle \eqref{cocycle rep 10943} and certain Bockstein spectral sequence differentials in section 5.1 of \cite{MR860042} if one wants to express $v_1^n\alpha_1$ in terms of the divided alpha element $\alpha_{n/(\nu_p(n)+1)}$.\footnote{Since this paper may have readers who are number theorists and not topologists, we remark that at each prime $p$ the {\em divided alpha-family} is the first and shortest-period of an infinite family of quasiperiodic families---the beta-family, the gamma-family, etc.---of $p$-power-torsion elements in the stable homotopy groups of spheres. These do not exhaust the stable homotopy groups of spheres, but they play an important role, partly because the $n$th Greek letter family is the $2p^n(p-1)$-quasiperiodic family which appears on the lowest possible line (i.e., the $s=n$ line) in the Adams-Novikov spectral sequence, and partly because the $n$th Greek letter family (as well as other $2p^n(p-1)$-quasiperiodic families in the Adams-Novikov $E_2$-term) is computable from the cohomology of the automorphism group scheme of a height $n$ formal group over $\mathbb{F}_p$, via some (very highly nontrivial) spectral sequence calculations. Chapters 5 and 6 of \cite{MR860042} are standard for this material.}



Assume $j<1+\nu_p(n)$, so that $p^jv_1^n\alpha_1\neq 0\in \pi_{2(p-1)(n+1)-1}(S^0)$.
By Proposition \ref{ahss splitting} (or just by analysis of the long exact sequence induced in $BP$-homology by the cofiber sequence $S^0 \rightarrow \cof p^jv_1^n\alpha_1\rightarrow S^{(2p-2)n}$), we have an isomorphism $BP_*(\cof p^jv_1^n\alpha_1) \cong BP_*\oplus \Sigma^{(2p-2)n}BP_*$ of graded $BP_*$-modules. In order to calculate the Hecke action on $\elll_*(\cof p^jv_1^n\alpha_1)$, we will need to know something about the $BP_*BP$-coaction\footnote{It is standard that, whenever $E$ is a ring spectrum and $X$ a spectrum, we have a coaction of $E_*E$ on $E_*X$ given by applying $\pi_*$ to the map of spectra $E\smash X \stackrel{\id_E \smash \eta\smash \id_X}{\longrightarrow} E\smash E\smash X$, where $\eta: S\rightarrow E$ is the unit map of the ring spectrum $E$.} $\psi: BP_*(\cof p^jv_1^n\alpha_1) \rightarrow BP_*BP\otimes_{BP_*}BP_*(\cof p^jv_1^n\alpha_1)$. We write $\tilde{\psi}$ for the composite of the coaction map with the evident isomorphisms:
\begin{align*} 
 \tilde{\psi}: BP_*\oplus \Sigma^{(2p-2)n}BP_* &\stackrel{\cong}{\longrightarrow} BP_*(\cof p^jv_1^n\alpha_1) \\
 &\stackrel{\psi}{\longrightarrow} BP_*BP \otimes_{BP_*} BP_*(\cof p^jv_1^n\alpha_1) \\
 &\stackrel{\cong}{\longrightarrow}  BP_*BP\oplus \Sigma^{(2p-2)n}BP_*BP.\end{align*}
Inclusion of the $0$-skeleton $S^0\rightarrow \cof p^jv_1^n\alpha_1$ induces the evident map of graded $BP_*BP$-comodules $BP_*\rightarrow BP_*(\cof p^jv_1^n\alpha_1)$ which picks out the left-hand summand in $BP_*(\cof p^jv_1^n\alpha_1)\cong BP_*\oplus \Sigma^{(2p-2)n}BP_*$. Hence $\tilde{\psi}(x,0) = (x,0)$ for all $x\in BP_*$, since the left unit\footnote{By convention, given an element $a\in A$, we also write $a$ to denote the element $\eta_L(a)$ in $\Gamma$. That is, we treat the left unit map $\eta_L$ as a canonical embedding of $A$ into $\Gamma$.} map $\eta_L: BP_* \rightarrow BP_*BP$ is the $BP_*BP$-coaction map on $BP_*(S^0)\cong BP_*$.

In the cobar complex \begin{equation}\label{cobar cplx 1} \xymatrix{
0 \ar[r]\ar[d] & 0 \ar[d] \\
BP_*(\cof p^jv_1^n\alpha_1) \ar[d]^{d^0}\ar[r]^{\cong} &  BP_*\oplus \Sigma^{(2p-2)n}BP_* \ar[d]^{\tilde{d}^0} \\
BP_*BP \otimes_{BP_*} BP_*(\cof p^jv_1^n\alpha_1) \ar[d]^{d^1}\ar[r]^{\cong} & BP_*BP \oplus \Sigma^{(2p-2)n}BP_*BP \ar[d]^{\tilde{d}^1} \\
BP_*BP^{\otimes_{BP_*} 2} \otimes_{BP_*} BP_*(\cof p^jv_1^n\alpha_1)\ar[d]^{d^2}\ar[r]^{\cong} & BP_*BP^{\otimes_{BP_*}2} \oplus \Sigma^{(2p-2)n}BP_*BP^{\otimes_{BP_*}2}\ar[d]^{\tilde{d}^2} \\
 \vdots  & \vdots ,
}\end{equation}
the first two differentials 
are as follows (see \cref{appendix on cobar complexes}):
\begin{align*}
 d^0(x) &= 1\otimes x - \psi(x) \\
 d^1(y\otimes x) &= 1\otimes y\otimes x - \Delta(y)\otimes x + y\otimes \psi(x),\mbox{\ \ i.e.,}\\
 \tilde{d}^0(x_0,x_1) &= \left(\eta_R(x_0) - \eta_L(x_0) - \zeta(x_1),\eta_R(x_1) - \eta_L(x_1)\right) \\
 \tilde{d}^1(y_0,y_1) &= \left(1\otimes y_0 - \Delta(y_0) + y_0\otimes 1 - \zeta(y_1), 1\otimes y_1 - \Delta(y_1) + y_1\otimes 1\right) 
\end{align*}
where $\zeta$ is the twist arising from the nontriviality of $p^jv_1^n\alpha_1$ when $j\leq \nu_p(n)$.
Taking the cofiber of $p^jv_1^n\alpha_1: S^{(2p-2)n-1}\rightarrow S^0$ kills the element $p^jv_1^n\alpha_1\in \pi_{(2p-2)n-1}(S^0)$ represented in the Adams-Novikov $E_2$-term by the element $\left[ p^j\sum_{i=1}^n \binom{n}{i} p^{i-1-\nu_p(n)}  v_1^{n-i} t_1^{i}\right]$ of  $\Ext^{1,(2p-2)n}_{BP_*BP-comod}(BP_*,BP_*)$, so the $1$-cocycle $\left(p^j\sum_{i=1}^n \binom{n}{i} p^{i-1-\nu_p(n)}  v_1^{n-i} t_1^{i},0\right)\in BP_*BP\oplus \Sigma^{(2p-2)n}BP_*BP$ must be a coboundary.
Since $j<1+\nu_p(n)$, for degree reasons this is only possible if $\zeta(1) = \left(p^j\sum_{i=1}^n \binom{n}{i} p^{i-1-\nu_p(n)}  v_1^{n-i} t_1^{i},0\right)$. (Of course, if $j\geq 1+\nu_p(n)$, we simply have $\zeta(1) =0$.)


Now it is possible to continue in a totally $p$-local way, by constructing a $p$-local splitting of elliptic cohomology much like the $p$-local splitting of $MU$ into a wedge of copies of $BP$, but eventually we want to return to a more global picture, since we want to understand Hecke eigenforms in $\elll_*(\cof p^jv_1^n\alpha_1)$---not some summand in a local splitting of $\elll_*(\cof p^jv_1^n\alpha_1)$---in terms of classical Hecke eigenforms.
So we return to working with $MU_{(p)}$, rather than $BP$.
Proposition \ref{ahss splitting} applies to $MU_{(p)}$ just as well as to $BP$,
so $(MU_{(p)})_*(\cof p^jv_1^n\alpha_1) \cong (MU_{(p)})_* \oplus \Sigma^{(2p-2)n}(MU_{(p)})_*$. The attaching map $p^jv_1^n\alpha_1$ is represented by a $1$-cocycle $p^j\sum_{i=1}^n \binom{n}{i} p^{i-1-\nu_p(n)}  v_1^{n-i} t_1^{i}$ in the image of the composite of the forgetful map $MU_*MU\rightarrow BP_*BP$ with the ring map $C: BP_*BP\rightarrow (MU_{(p)})_*MU_{(p)}$ 
which classifies the Cartier $p$-typicalization map of the universal strict isomorphism of $1$-dimensional formal group laws.
Consequently the composite map
\begin{align*} 
 \tilde{\psi}_{MU}: (MU_{(p)})_*\oplus \Sigma^{(2p-2)n}(MU_{(p)})_* 
  &\stackrel{\cong}{\longrightarrow} (MU_{(p)})_*(\cof p^jv_1^n\alpha_1) \\
  &\stackrel{\psi}{\longrightarrow} (MU_{(p)})_*MU_{(p)} \otimes_{(MU_{(p)})_*} (MU_{(p)})_*(\cof p^jv_1^n\alpha_1) \\
  &\stackrel{\cong}{\longrightarrow}  (MU_{(p)})_*MU_{(p)}\oplus \Sigma^{(2p-2)n}(MU_{(p)})_*MU_{(p)}\end{align*}
sends $(x_0,x_1)$ to $\left(x_0 + x_1C\left(p^j\sum_{i=1}^n \binom{n}{i} p^{i-1-\nu_p(n)}  v_1^{n-i} t_1^{i}\right),x_1\right)$.

Now suppose we have a pair $(f_k,f_{k-(p-1)n})$ of modular forms of level 1, of weights $k$ and $k-(p-1)n$ respectively. We regard $(f_k,f_{k-(p-1)n})$ as an element of 
\[  \elll_{2k}(\cof p^jv_1^n\alpha_1) \cong \left(\elll_* \oplus \Sigma^{(2p-2)n}\elll_*\right)_{2k} \cong \elll_{2k}\oplus \elll_{2(k-(p-1)n)}.\]
The $MU_{(p)}$-coaction on $\elll_{*}(\cof p^jv_1^n\alpha_1)$ sends $(f_k,f_{k-(p-1)n})$ to 
\[ \left(f_k + f_{k-(p-1)}C\left(p^j\sum_{i=1}^n \binom{n}{i} p^{i-1-\nu_p(n)}  v_1^{n-i} t_1^{i}\right),f_{k-(p-1)n}\right).\] This tells us the behavior of the map ``$1\otimes\overline{mu}$,'' in the case $X = \cof p^jv_1^n\alpha_1$,
under the composite
\begin{align*} 
 \elll_*(X) 
  &\stackrel{\cong}{\longrightarrow} \elll_*\otimes_{MU_*}MU_*(X) \\
  &\stackrel{1\otimes\overline{mu}}{\longrightarrow} \elll_*\otimes_{MU_*}MU_*MU \otimes_{MU_*}MU_*(X) \\
  &\stackrel{\overline{H}\otimes 1}{\longrightarrow} \elll_*\otimes_{MU_*}MU_*(X) \\
  &\stackrel{\cong}{\longrightarrow} \elll_*(X) 
\end{align*}
which Baker uses to define the Hecke action on elliptic homology of a spectrum $X$ (\cite{MR1037690}, bottom of p. 9). Note that Baker works with elliptic cohomology, not elliptic homology, but to switch between the two is entirely formal. The map $\overline{H}$ depends on a choice of prime $\ell$, and it can be calculated as follows: write any element $f\otimes u \in \elll_*\otimes_{MU_*}MU_*(X)$ as $\sum_n f_n\otimes u_n$, where $u_n \in \eta_R(MU_*\otimes_{\mathbb{Z}} \mathbb{Q})$ and $f_n \in \elll_*\otimes_{\mathbb{Z}} \mathbb{Q}$. Then (\cite{MR1037690}, middle of p. 9)
\begin{equation}\label{baker eq 2} 
\overline{H}(f_n\otimes v_n) = T_{\ell}(f_n)\otimes v_n
\end{equation}
where $T_{\ell}$ is a classical Hecke operator on level one modular forms.

So, if we choose some prime $\ell$, we have a calculation of the topological Hecke operator $\tilde{T}_{\ell}$ on the elliptic homology of $\cof p^jv_1^n\alpha_1$. Given $(x_0,y_0)\in \elll_*\oplus \Sigma^{(2p-2)n}\elll_*\cong \elll_*(\cof p^jv_1^n\alpha_1)$, we have:
\begin{description}
\item[If $p=\ell$] One has to invert $\ell$ to get Baker's action of $\tilde{T}_{\ell}$ on elliptic homology. Of course $\elll_*(-)[1/p] \cong \pi_*\left(\elll \smash - \smash S^0[1/p]\right)$, so we have $\cof p^jv_1^n\alpha_1 \smash S^0[1/p] \simeq S^0[1/p]\vee S^{(2p-2)n}[1/p],$
since the attaching map \[ p^jv_1^n\alpha_1\in \pi_{(2p-2)(n+1)-1}(S^0)\] is $p$-power-torsion and hence becomes nulhomotopic on inverting $p$. So $\tilde{T}_p$ acts on $\elll_*(\cof p^jv_1^n\alpha_1)[1/p]$ the same way that $\tilde{T}_p$ acts on $\elll_*(S^0\vee S^{(2p-2)n})[1/p]$, i.e., diagonally by the classical Hecke action. 

So a weight $k$ eigenform for $\tilde{T}_p$ on $\cof p^jv_1^n\alpha_1$ is simply a pair $(f_k,f_{k-(p+1)n})$ of classical modular forms, with $f_k$ weight $k$ and $f_{k-(p+1)n}$ weight $k-(p+1)n$, such that $f_k$ and $f_{k-(p+1)n}$ are each eigenforms for $T_p$ in the classical sense, with the same eigenvalue. (The density argument of Proposition \ref{prop:HH0} then implies that, if $(f_k,f_{k-(p+1)n})$ is a topological Hecke eigenform for all but finitely many primes, then either $f_k$ or $f_{k-(p+1)n}$ must be zero; this is the same argument as in Proposition \ref{examples of hecke eigenforms}.)
\item[If $p\neq \ell$]
\begin{align}
\nonumber T_{\ell}(f_0,f_1) 
  &= \left(\overline{H}\otimes 1\right)(f_0\otimes 1 + f_1\otimes p^j\sum_{i=1}^n \binom{n}{i} p^{i-1-\nu_p(n)}  v_1^{n-i} t_1^{i},f_1\otimes 1) \\
\nonumber  &= \left(\overline{H}\otimes 1\right)\left(f_0\otimes 1 + p^j f_1\otimes \sum_{i=1}^n \binom{n}{i} p^{i-1-\nu_p(n)}  \eta_L(v_1^{n-i}) \left( \frac{1}{p}\eta_R(v_1) - \frac{1}{p}\eta_L(v_1)\right)^{i},f_1\otimes 1\right) \\
\nonumber  &= \left(\overline{H}\otimes 1\right)\left(f_0\otimes 1 + p^j f_1\otimes \sum_{i=1}^n \binom{n}{i} p^{-1-\nu_p(n)}  \eta_L(v_1^{n-i}) \sum_{k=0}^i (-1)^k\binom{i}{k}\eta_R(v_1)^{i-k} \eta_L(v_1)^k,f_1\otimes 1\right) \\
\nonumber  &= \left( T_{\ell}(f_0)    
  + \sum_{a=0}^n \sum_{i=\max\{1,a\}}^n p^{j-1}\binom{n}{i}\binom{i}{i-a}\frac{(-1)^{i-a}}{n}  T_{\ell}(f_1 v_1^{n-a}) \xi(\eta_R(v_1)^{a}),
   T_{\ell}(f_1)\right) \\
\label{eq fj049344}  &= \left( T_{\ell}(f_0) 
  + p^{j-1}\sum_{a=0}^n \frac{(-1)^a}{n}\sum_{i=\max\{1,a\}}^n (-1)^{i} \binom{n}{i} \binom{i}{i-a}  T_{\ell}(f_1 v_1^{n-a}) \xi(\eta_R(v_1)^{a}),
   T_{\ell}(f_1)\right) .
\end{align}
\end{description}

\begin{theorem}\label{eigenform conditions thm}
Let $j$ and $n$ be nonnegative integers with $j\leq \nu_p(n)$, and let $X$ denote the homotopy cofiber of $p^jv_1^n\alpha_1 \in \pi_{(2p-2)(n+1) - 1}(S^0)$. Let $P$ be a cofinite set of prime numbers not containing $p$.
Suppose we are given a finite extension $K/\mathbb{Q}$ and 
an eigencharacter $\lambda: \mathbf{A}_P \rightarrow \mathcal{O}_K[P^{-1}]$.
Given holomorphic modular forms $(f_0,f_1)$ over $\mathcal{O}_K[P^{-1}]$ of level $1$ and weights $k$ and $k-(p-1)n$ respectively, the element $(f_0,f_1) \in \left(M_{k}\oplus M_{k-(p-1)n}\right)\otimes_{\mathbb{Z}}\mathcal{O}_K[P^{-1}] \cong \elll_{2k}(X)\otimes_{\mathbb{Z}}\mathcal{O}_K[P^{-1}]$ is a topological eigenform over $X$ for $P$ with eigencharacter $\lambda$ if and only if
\begin{align}
\label{eigenform condition 1} T_{\ell}(f_1) &= \lambda(\ell)f_1, \mbox{\ \ and} \\
\label{eigenform condition 2} 
T_{\ell}(f_0) + p^{j-1-\nu_p(n)} \left(T_{\ell}(f_1E_{p-1}^n) - T_{\ell}(f_1)E_{p-1}^n \right) &= \lambda(\ell)f_0. 
%
\end{align}
for all $\ell \in P$. 
\end{theorem}
\begin{proof}
Theorem \ref{multiplicity one thm} gives us that $X$ has multiplicity one for $P$ (see Definition \ref{def of multiplicity one} for the multiplicity one condition for topological Hecke eigenforms), so by Proposition \ref{multiplicity one is convenient}, if $(f_0,f_1)$ is an eigenform for the action of $\tilde{T}_{\ell}$ for each prime $\ell\in P$, then $(f_0,f_1)$ is a topological Hecke eigenform for $P$. 

From \eqref{eq fj049344} and \eqref{baker eq 2} we have that 
\begin{align}
\nonumber \tilde{T}_{\ell}(f_0,f_1) 
  &= \left( T_{\ell}(f_0)  + \sum_{a=0}^n \sum_{i=\max\{1,a\}}^n p^{j-1}\binom{n}{i}\binom{i}{i-a}\frac{(-1)^{i-a}}{n}  T_{\ell}(f_1 E_{p-1}^{n-a}) E_{p-1}^{a},T_{\ell}f_1\right) \\
\label{eq ok409409}  &= \left( T_{\ell}(f_0)  + \left(\frac{p^{j-1}}{n} \sum_{i=1}^n\binom{n}{i}(-1)^i T_{\ell}(f_1 E_{p-1}^n)\right) + \frac{p^{j-1}}{n}T_{\ell}(f_1)E_{p-1}^n,T_{\ell}f_1\right) \\
\label{eq ok409410}  &= \left( T_{\ell}(f_0)  + \frac{p^{j-1}}{n}\left(  T_{\ell}(f_1)E_{p-1}^n- T_{\ell}(f_1 E_{p-1}^n) \right),T_{\ell}f_1\right) 
\end{align}
with \eqref{eq ok409409} due to Lemma \ref{binomial coeffs lemma}, and with \eqref{eq ok409410} due to the equality $-1 = \sum_{i=1}^n \binom{n}{i} (-1)^i$, a simple consequence of expanding out $(x-1)^n$ and then substituting $1$ for $x$.
\end{proof}

\begin{corollary}\label{main thm cor 123409}
Let $j$ and $n$ be nonnegative integers with $j\geq \nu_p(n)$, and let $X$ denote the homotopy cofiber of $p^jv_1^n\alpha_1 \in \pi_{(2p-2)(n+1) - 1}(S^0)$.
Suppose we are given a finite extension $K/\mathbb{Q}$ and a set $P$ of prime numbers.
Let $f$ be a weight $k-(p-1)n$ level $1$ eigenform, over $\mathcal{O}_K[P^{-1}]$, for $T_{\ell}$ for all primes $\ell\in P$. Let $\lambda: A \rightarrow \mathcal{O}_K[P^{-1}]$ be the eigencharacter of $f$.
 Then there exists a topological Hecke eigenform on $X$ whose restriction to the top cell is $f$ if and only if the derived Hecke eigenform $f\cupdot p^j \kappa^{E_{p-1}}_n\in HH^1\left(A; M_{k}^{\lambda}\right)$ is trivial.
\end{corollary}
\begin{proof}
Unwinding the definitions and comparing the parenthesized terms in \eqref{eigenform condition 2} to the derivation $\Delta_n$ defined in Definition-Proposition \ref{def of delta} (which was used, in turn, to define the cocycle representative $\phi_n^E$ for $\kappa^E_n$ in Definition-Proposition \ref{def of kappas}), the condition \eqref{eigenform condition 2} expresses precisely that the Hochschild $1$-cocycle $p^jf_1 \cupdot\phi^{E_{p-1}}_n$ is the coboundary of $f_0$.
\end{proof}

\begin{prop}\label{nonvanishing of dotcup 0}
Let $n$ be a nonnegative integer.
Suppose we are given a finite extension $K/\mathbb{Q}$ and a cofinite set of prime numbers $P$.
Suppose $p$ is a prime number not contained in $P$, and suppose that $f$ is a holomorphic Hecke eigenform over $\mathbb{Z}[P^{-1}]$ of level $1$ and weight $k$ with eigencharacter $\lambda: A\rightarrow \mathcal{O}_K[P^{-1}]$.
If $f$ is not divisible by $p$, then the derived Hecke eigenform $f\cupdot \kappa^{E_{p-1}}_n\in HH^1(A; M_k^{\lambda})$ has (additive) order $p^{1+\nu_p(n)}$.
\end{prop}
\begin{proof}
In Theorem \ref{nontriviality of all kappas} we showed that $\kappa^{E_{p-1}}_n\in HH^1\left(A; \hom_{\mathbb{Z}[P^{-1}]}(M_{*-(p-1)n},M_{*})\right)$ is of order $p^{1+\nu_p(n)}$, so the order of $f\cupdot \kappa^{E_{p-1}}_n$ must be a divisor of $p^{1+\nu_p(n)}$. 
Suppose $p^j\left(f\cupdot \kappa^{E_{p-1}}_n\right) = 0$. 
Then $p^jf\cupdot \phi^{E_{p-1}}_n$ must be a coboundary, i.e., there must exist some holomorphic modular form $h\in M_k^{\lambda}$ such that equality \eqref{eq 0439jfj} holds in the chain of equalities
\begin{align}
 T_{\ell}(h) - \lambda(\ell)h 
  &= (dh)(T_{\ell}) \\
\label{eq 0439jfj}  &= \left(p^jf\cupdot \phi^{E_{p-1}}_n\right)(T_{\ell}) \\
\nonumber  &= p^{j-\nu_p(n)-1} \left( T_{\ell}(E_{p-1}^nf) - E_{p-1}^n T_{\ell}(f)\right) \\
\nonumber  &= p^{j-\nu_p(n)-1} \left( T_{\ell}(E_{p-1}^nf) - E_{p-1}^n \lambda_{\ell}f\right) ,\mbox{\ \ \ i.e.,}\\
\label{eq 0439jfk} T_{\ell}(h - p^{j-\nu_p(n)-1}E_{p-1}^nf) &= \lambda(\ell) \left( h - p^{j-\nu_p(n)-1}E_{p-1}^n f\right).
\end{align}
Equality \eqref{eq 0439jfk} establishes that $h - p^jE_{p-1}^nf$ and $f$ are each Hecke eigenforms (for all primes in the cofinite set $P$) of level $1$ and of weights $k$ and $k-2(p-1)n$, respectively. The density argument from Proposition \ref{prop:HH0} consequently implies that either $f=0$ or $h - p^{j-\nu_p(n)-1}E_{p-1}^nf=0$. Since we assumed that $f$ is not divisible by $p$, we cannot have $f=0$. So we must have $h = p^{j-\nu_p(n)-1}E_{p-1}^nf$. Now $f$ is not divisible by $p$, and $E_{p-1}^n\equiv 1$ modulo positive degree terms in its $q$-expansion, so $fE_{p-1}^n$  is also not divisible by $p$, i.e., $fE_{p-1}^n$ is not equal to a positive power of $p$ times any modular form, in particular, $h$. So $j-\nu_p(n)-1\geq 0$, i.e., $j\geq \nu_p(n)+1$. So the only way for the cohomology class $p^j\left(f\cupdot \kappa^{E_{p-1}}_n\right)$ to be zero is for $j$ to be at least $\nu_p(n)+1$, i.e., the order of $f\cupdot \kappa^{E_{p-1}}_n$ is $\nu_p(n)+1$.
\end{proof}

We are now ready to give a complete description of the topological Hecke eigenforms over $2$-cell complexes whose attaching map has $BP$-Adams degree $1$. In Theorem \ref{main thm 304jf}, the elliptic homology $\elll_{*}(X)$ vanishes in odd degrees, so the description given of topological Hecke eigenforms over $X$ in even degrees is indeed a description of all topological Hecke eigenforms over $X$.
\begin{theorem}\label{main thm 304jf}
Let $n$ be a nonnegative integer.
Suppose we are given a cofinite set of prime numbers $P$. Let $p$ be an odd prime, let $k$ be an integer, and let $X$ be the cofiber of $p^jv_1^n\alpha_1$ for some integer $j$. 
If $j > \nu_p(n)$, then $p^jv_1^n\alpha_1\in \pi_{2(p-1)n-1}(S^0)$ is nulhomotopic, so $X$ is stably homotopy equivalent to a wedge of spheres, and so its topological Hecke eigenforms are described by Proposition \ref{examples of hecke eigenforms}. If $j\leq \nu_p(n)$, then the topological Hecke eigenforms over $X$ are of exactly two types:
\begin{description}
\item[Topological Hecke eigenforms supported on the bottom cell] under the splitting 
\[\elll_{2k}(X)\otimes_{\mathbb{Z}}\mathcal{O}_{\Qbar}[P^{-1}]
\cong (M_k \oplus M_{k-(p-1)n}) \otimes_{\mathbb{Z}}\mathcal{O}_{\Qbar}[P^{-1}],\] 
the topological Hecke eigenforms in $\elll_{2k}(X)\otimes_{\mathbb{Z}}\mathcal{O}_{\Qbar}[P^{-1}]$ corresponding to an element $(f,0)\in (M_k \oplus M_{k-(p-1)n}) \otimes_{\mathbb{Z}}\mathcal{O}_{\Qbar}[P^{-1}]$ with $f$ a classical Hecke eigenform; and
\item[Topological Hecke eigenforms nontrivial on the top cell] 
topological Hecke eigenforms of the form $(-p^{j-1-\nu_p(n)}gE_{p-1}^n,g)\in (M_k \oplus M_{k-(p-1)n}) \otimes_{\mathbb{Z}}\mathcal{O}_{\overline{\mathbb{Q}}}[P^{-1}]$ with $g$ a classical Hecke eigenform divisible by $p^{1+\nu_p(n)-j}$.
\end{description}
\end{theorem}
\begin{proof} 
Suppose $(f,g)\in (M_k \oplus M_{k-(p-1)n}) \otimes_{\mathbb{Z}}\mathcal{O}_{\Qbar}[P^{-1}]$ defines a topological Hecke eigenform over $X$, with $g$ nonzero. By Theorem \ref{eigenform conditions thm}, we must have that $g$ is a classical Hecke eigenform, i.e., $g\in HH^0(A; M_{k-(p-1)n}^{\lambda})$ where $\lambda$ is the eigencharacter of $g$. By Corollary \ref{main thm cor 123409}, we must have that the derived Hecke eigenform $g\cupdot p^j\kappa^{E_{p-1}}_n\in HH^1(A; M_k^{\lambda})$ vanishes. By Proposition \ref{nonvanishing of dotcup 0}, we have $g\cupdot p^j\kappa^{E_{p-1}}_n=0$ if and only if $g$ is divisible by $p^{1+\nu_p(n)-j}$.
However, if $g$ is divisible by $p^{1+\nu_p(n)-j}$, then \eqref{eigenform condition 2} reads
\begin{align*}
 T_{\ell}(f) + p^{j-1-\nu_p(n)} \left(T_{\ell}(gE_{p-1}^n) -  \lambda_{\ell}(g)E_{p-1}^n \right)&= \lambda(\ell)f,\mbox{\ \ \ i.e.,}\\
 T_{\ell}(f + p^{j-1-\nu_p(n)}gE_{p-1}^n) &= \lambda(\ell)(f + p^{j-1-\nu_p(n)}gE_{p-1}^n),
\end{align*}
i.e., $g$ and $f + p^{j-1-\nu_p(n)}gE_{p-1}^n$ are Hecke eigenforms over $\mathbb{Z}[P^{-1}]$ of the same level and eigencharacter but with differing weights. The density argument of Proposition \ref{prop:HH0} gives us that either $g$ is zero (which is false by assumption) or $f + p^{j-1-\nu_p(n)}gE_{p-1}^n$ is zero, i.e., $f = -p^{j-1-\nu_p(n)}gE_{p-1}^n$.
\end{proof}

\begin{remark}
It seems extremely plausible that the comparison of the Hochschild cohomology $HH^1(A; \hom_{\mathbb{Z}[P^{-1}]}(M,M))\cong \Ext^1_A(M,M)$ of the ring of modular forms to the Adams-Novikov $1$-line, carried out in this paper, is related to the results of \cite{MR1692001} by some kind of global duality which would presumably exchange Hecke invariants in $\Ell_*(S^0)\otimes_{\mathbb{Z}} \mathbb{Q}/\mathbb{Z}$ and derived Hecke eigenforms in $\Ell_*(S^0)$. In \cite{MR1692001}, Baker calculates fixed points of the Hecke action on $\Ell_*(S^0)\otimes_{\mathbb{Z}}\mathbb{Q}/\mathbb{Z}$ in order to get the coinvariants $\Ext^{0,*}_{\Ell_*\Ell-comod}\left(\Ell_*(S^0),\Ell_*(S^0)\otimes_{\mathbb{Z}}\mathbb{Q}/\mathbb{Z}\right)$ of the coaction of the Hopf algebroid $(\Ell_*(S^0),\Ell_*\Ell)$ of stable co-operations in weakly holomorphic elliptic homology on $\Ell_*(S^0)\otimes_{\mathbb{Z}}\mathbb{Q}/\mathbb{Z}$. The long exact sequence in comodule $\Ext$ induced by the extension 
\[ 0 \rightarrow \Ell_*(S^0) \rightarrow \Ell_*(S^0)\otimes_{\mathbb{Z}}\mathbb{Q} \rightarrow \Ell_*(S^0)\otimes_{\mathbb{Z}}\mathbb{Q}/\mathbb{Z} \rightarrow 0\]
then lets Baker calculate the $1$-line $\Ext^{1,*}_{\Ell_*\Ell-comod}\left(\Ell_*(S^0),\Ell_*(S^0)\right)$ in the $\Ell$-Adams spectral sequence $\Ext^{s,t}_{\Ell_*\Ell-comod}\left(\Ell_*(S^0),\Ell_*(S^0)\right) \Rightarrow \pi_{t-s}(L_{\Ell}S)$, deducing that the orders of the groups in the $\Ell$-Adams $1$-line agree with the orders of the groups in the $MU[\frac{1}{6}]$-Adams $1$-line (i.e., the $6$-inverted Adams-Novikov $1$-line). Baker's approach does not, however, result in a description of the effect of Hecke operators on cofibers of elements in the Adams-Novikov $1$-line, or the topological Hecke eigenforms over such cofibers; so the results of \cite{MR1692001} do not overlap greatly with the results of the present paper, although we think it is likely that some of the results of the present paper could also be proven by developing some appropriate kind of duality and applying it to the results of \cite{MR1692001}. Such an approach would not result in the connection to the Hochschild-cohomological ``derived eigentheory'' from \cref{Abstract calculus....} that we establish in Corollary \ref{main thm cor 123409}, however.
\end{remark}

\subsection{The spectral sequence for derived topological Hecke eigenforms.}

Let $X$ be a topological space, let $P$ be a set of primes, and let $\lambda: \mathbf{A}_P \rightarrow \mathcal{O}_K[P^{-1}]$ be an eigencharacter. Combining Definition \ref{def of eigenform} and the definition of a derived Hecke eigenform immediately following Definition \ref{bimodules definition 1}, by a {\em derived topological Hecke eigenform on $X$ with eigencharacter $\lambda$} we mean an element of 
\[ HH^*\left(\mathbf{A}_P; (\elll_*(X)\otimes_{\mathbb{Z}}\mathcal{O}_K[P^{-1}])^{\lambda}\right).\] Then, for example, $HH^0\left(\mathbf{A}_P; (\elll_*(X)\otimes_{\mathbb{Z}}\mathcal{O}_K[P^{-1}])^{\lambda}\right)$ coincides with the (non-derived) topological Hecke eigenforms over $X$ with eigencharacter $\lambda$.

We have a spectral sequence which goes from derived Hecke eigenforms to derived topological Hecke eigenforms:
\begin{theorem} 
There exists a trigraded spectral sequence
\begin{align}
\label{sseq 04394} E_1^{s,t,k} \cong H_t(X; \mathbb{Z})\otimes_{\mathbb{Z}} HH^s(\mathbf{A}_P; M_k^{\lambda}) &\Rightarrow HH^s\left(\mathbf{A}_P; (\elll_k(X)\otimes_{\mathbb{Z}}\mathcal{O}_K[P^{-1}])^{\lambda}\right)\\
\nonumber d_r: E_r^{s,t,k} &\rightarrow E_r^{s+1,t-r,k+r}
\end{align}
which is strongly convergent if $X$ is a finite-dimensional CW-complex.
\end{theorem}
\begin{proof}
Equip $\left(\elll_*(X)\otimes_{\mathbb{Z}}\mathcal{O}_K[P^{-1}]\right)^{\lambda}$ with the skeletal (i.e., Atiyah-Hirzebruch) filtration 
\begin{equation}\label{filt bimodule 2}\left(\elll_*(X^0)\otimes_{\mathbb{Z}}\mathcal{O}_K[P^{-1}]\right)^{\lambda}\subseteq \left(\elll_*(X^1)\otimes_{\mathbb{Z}}\mathcal{O}_K[P^{-1}]\right)^{\lambda}\subseteq \left(\elll_*(X^2)\otimes_{\mathbb{Z}}\mathcal{O}_K[P^{-1}]\right)^{\lambda}\subseteq \dots\end{equation} and then take the cyclic cobar complex of $\mathbf{A}_P$ with coefficients in each of these bimodules \eqref{filt bimodule 2}, to get a filtered cochain complex whose cohomology is \[ HH^*\left(\mathbf{A}_P; (\elll_k(X)\otimes_{\mathbb{Z}}\mathcal{O}_K[P^{-1}])^{\lambda}\right).\] and whose associated graded is \[ H_*(X;\mathbb{Z})\otimes_{\mathbb{Z}}HH^*\left(\mathbf{A}_P; (\elll_*(S^0)\otimes_{\mathbb{Z}}\mathcal{O}_K[P^{-1}])^{\lambda}\right).\] Spectral sequence \eqref{sseq 04394} is simply the spectral sequence of this filtered cochain complex.
\end{proof}

In the case where $X$ is the two-cell complex $\cof p^jv_1^n\alpha_1$, spectral sequence \eqref{sseq 04394} is concentrated on the $t=0$ and $t=2n(p-1)$ lines already at the $E_1$-term, and so the only possible nonzero differentials are $d_{2n(p-1)}$-differentials. Consequently the spectral sequence degenerates to a Gysin-type long exact sequence
\begin{equation*}
\xymatrix{
  0 \rightarrow HH^0(A; M_k^{\lambda}) \ar[r] & 
  HH^0(A;\elll_k(X)^{\lambda}) \ar[r] &
  HH^0(A; M_{k-2n(p-1)}^{\lambda}) \ar`r_l[ll] `l[dll]^{d_{2n(p-1)}} [dll] \\
  HH^1(A; M_k^{\lambda}) \ar[r] &
  HH^1(A;\elll_k(X)^{\lambda}) \ar[r] &
  HH^1(A; M_{k-2n(p-1)}^{\lambda}) \ar`r_l[ll] `l[dll]^{d_{2n(p-1)}} [dll] \\
  HH^2(A; M_k^{\lambda}) \ar[r] &
  HH^2(A;\elll_k(X)^{\lambda})  \ar[r] &
  \dots }\end{equation*}
for each integer $k$.
It follows from Corollary \ref{main thm cor 123409} that the differential 
\[ HH^0(A; M_{k-2n(p-1)}^{\lambda}) \stackrel{d_{2n(p-1)}}{\longrightarrow} HH^1(A; M_k^{\lambda})\] is simply the dot-cup product $f \mapsto f\cupdot p^j \kappa^{E_{p-1}}_n$. Proposition \ref{nonvanishing of dotcup 0} then establishes that this differential sends an additive generator of $HH^0(A; M_{k-2n(p-1)}^{\lambda})$ to an element of order $p^{1+\nu_p(n)}$ in $HH^1(A; M_k^{\lambda})$.

We have not tried to determine whether the higher differentials in this spectral sequence are also given by the dot-cup product with $p^j \kappa^{E_{p-1}}_n$.

\appendix
\section{Appendix on cobar complexes.}
\label{appendix on cobar complexes}

In this paper, we occasionally must refer to cocycles in the cobar complex of a Hopf algebroid $(A,\Gamma)$ with coefficients in a left $\Gamma$-comodule $M$. The standard reference for Hopf algebroids is Appendix 1 of \cite{MR860042}, and the (two-sided) cobar complex of $(A,\Gamma)$ is defined in Definition A1.2.11 of \cite{MR860042}. For convenience, we recall the one-sided version of that definition here, as well as a simplification that occurs when the coefficient comodule $M$ is $A$ itself.
\begin{definition}
Let $(A,\Gamma)$ be a Hopf algebroid with left unit map $\eta_L:A \rightarrow \Gamma$, right unit map $\eta_R:A \rightarrow \Gamma$, and coproduct $\Delta: \Gamma\rightarrow \Gamma\otimes_A \Gamma$. Let $M$ be a left $\Gamma$-comodule with structure map $\psi: M \rightarrow \Gamma\otimes_A M$. 
Then the (one-sided) cobar complex of $(A,\Gamma)$ with coefficients in $M$ is the Moore complex\footnote{The ``Moore complex'' of a cosimplicial abelian group $X^{\bullet}$ is the alternating sign cochain complex of $X^{\bullet}$, i.e., the cochain complex whose group of $n$-cochains is $X^n$ and whose differential $X^n \rightarrow X^{n+1}$ is the alternating sum of the coface maps $X^n \rightarrow X^{n+1}$ in $X^{\bullet}$.}
\begin{equation}\label{cosimp 4039} C^{\bullet} = \left( \xymatrix{ 
M \ar@<1ex>[r] \ar@<-1ex>[r] & 
 \Gamma\otimes_A M  \ar[l]\ar@<2ex>[r]\ar[r]\ar@<-2ex>[r] & 
 \Gamma\otimes_A\Gamma\otimes_A M  \ar@<1ex>[l] \ar@<-1ex>[l] \ar@<3ex>[r] \ar@<1ex>[r] \ar@<-1ex>[r] \ar@<-3ex>[r] & \dots \ar@<2ex>[l]\ar[l]\ar@<-2ex>[l] , }\right)\end{equation}
whose $0$th coface map $C^n \rightarrow C^{n+1}$ is $\Delta\otimes_A \id_{\Gamma}\otimes_A \id_{\Gamma}\otimes_A \dots \otimes_A \id_{\Gamma}\otimes_A \id_M$,
whose $1$st coface map $C^n \rightarrow C^{n+1}$ is $\id_{\Gamma}\otimes_A\Delta\otimes_A\id_{\Gamma}\otimes_A \dots \otimes_A \id_{\Gamma}\otimes_A \id_M$, and so on, up through its
$(n-1)$st coface map $C^n \rightarrow C^{n+1}$ given by $\id_{\Gamma}\otimes_A \id_{\Gamma}\otimes_A \dots \otimes_A \id_{\Gamma}\otimes_A\Delta\otimes_A \id_M$,
and with its last (that is, $n$th) coface map $C^n \rightarrow C^{n+1}$ given by $\id_{\Gamma}\otimes_A \dots \otimes_A \id_{\Gamma}\otimes_A \psi$.

When $M = A$, the above simplifies, and \eqref{cosimp 4039} is isomorphic to the cosimplicial abelian group
\begin{equation}\label{cosimp 4039a} \tilde{C}^{\bullet} = \left( \xymatrix{ 
A \ar@<1ex>[r] \ar@<-1ex>[r] & 
 \Gamma  \ar[l]\ar@<2ex>[r]\ar[r]\ar@<-2ex>[r] & 
 \Gamma\otimes_A\Gamma  \ar@<1ex>[l] \ar@<-1ex>[l] \ar@<3ex>[r] \ar@<1ex>[r] \ar@<-1ex>[r] \ar@<-3ex>[r] & \dots \ar@<2ex>[l]\ar[l]\ar@<-2ex>[l] , }\right)\end{equation}
whose $0$th coface map $\tilde{C}^0 \rightarrow\tilde{C}^1$ is $\eta_R$, and 
whose $1$st coface map $\tilde{C}^0 \rightarrow\tilde{C}^1$ is $\eta_L$;
and, when $n>0$, whose $0$th coface map $\tilde{C}^n \rightarrow \tilde{C}^{n+1}$ sends $x_1\otimes \dots \otimes x_n$ to $1\otimes x_1\otimes \dots \otimes x_n$, whose $n$th coface map $\tilde{C}^n \rightarrow \tilde{C}^{n+1}$ sends $x_1\otimes \dots \otimes x_n$ to $x_1\otimes \dots \otimes x_n\otimes 1$, and whose intermediate coface maps $\tilde{C}^n \rightarrow \tilde{C}^{n+1}$ are given by letting the $1$st coface map be $\Delta \otimes_A \id_{\Gamma}\otimes_A \dots \otimes_A \id_{\Gamma}$, letting the $2$nd coface map be $\id_{\Gamma}\otimes_A\Delta \otimes_A \id_{\Gamma}\otimes_A \dots \otimes_A \id_{\Gamma}$, and so on.
\end{definition}
The above conventions are standard, and agree with the more well-known standard conventions when working with Hopf algebras. When working with Hopf algebroids, an extra wrinkle is introduced by having two unit maps: the convention when working with Hopf algebroids is that, when we tensor $\Gamma$ over $A$ on the left, we use the $A$-module structure on $\Gamma$ given by the left unit map $\eta_L: A\rightarrow \Gamma$, and when we tensor $\Gamma$ over $A$ on the right, we use the $A$-module structure on $\Gamma$ given by the right unit map $\eta_R:A\rightarrow \Gamma$. So, for example, since $\eta_L(v_1) = v_1$ and $\eta_R(v_1) = v_1 + pt_1$ in the Hopf algebroid $(BP_*,BP_*BP)$ of stable co-operations in Brown-Peterson homology, we have that \[ v_1 \otimes t_1 = (\eta_R(v_1) - pt_1)\otimes t_1 = 1\otimes v_1t_1\ -\ pt_1 \otimes t_1\in BP_*BP\otimes_{BP_*}BP_*BP\cong \tilde{C}^2.\] With these conventions in place, it is an exercise to verify that \eqref{cocycle rep 10943} is a $1$-cocycle in $\tilde{C}^{\bullet}$.


\end{document}